\documentclass[12pt]{article}

\usepackage{amssymb,amsthm,amscd,mathrsfs,amsbsy,amsfonts}
\usepackage{pstricks,pst-node}
\usepackage{amsbsy,young,colortbl,cite}

\usepackage{makecell}%长表格问题
\usepackage{longtable}
\usepackage{threeparttable}%表格注释
\usepackage{supertabular}%跨页表格
\usepackage{caption}%\captionof{figure}{sw2}
\usepackage{graphicx}%插入图片
\graphicspath{{figures/}} %如果图片文件不在代码文件目录下，可以用此代码规定图片路径
\usepackage[justification=centering]{caption}% 标题居中
\usepackage{subfigure}%并排放置多张图片%有此包，便不能再导入subfig包，两者不能同时存在
\usepackage{float}  %设置图片浮动位置的宏包

\usepackage{amssymb,amsthm}
\usepackage[english]{babel}

\usepackage{algorithm}
\usepackage{algorithmic}
\usepackage{multirow} %multirow for format of table

\usepackage{threeparttable}%表格注释

\usepackage{hyperref}
\hypersetup{hypertex=true,colorlinks=true,linkcolor=blue,citecolor=green}%anchorcolor=blue,

\newtheorem{proposition}{Proposition}
\newtheorem{lemma}{Lemma}
\newtheorem{definition}{Definition}
\newtheorem{theorem}{Theorem}

\newtheorem{remark}{Remark}
\newtheorem{example}{Example}

\newcommand{\T}{\mathfrak t}

\newcommand{\R}{\mathbb R}

\newcommand{\N}{\mathbb N}

\renewcommand{\T}{\mathbb{T}}

\def\({\left(}
\def\){\right)}
\def\[{\begin{eqnarray}}
\def\]{\end{eqnarray}}

\usepackage[leqno]{amsmath}%公式靠左
\numberwithin{equation}{section}

\usepackage{amssymb}

\begin{document}

\title{Similarity Between Two Stochastic Differential Systems %Based on Differential Equations
\footnotemark[1]}
\author{\  \  Xiaoying Wang \footnotemark[2], Yuecai Han \footnotemark[2], Yong Li \footnotemark[2] \footnotemark[3]}
\date{}
\maketitle
\renewcommand{\thefootnote}{\fnsymbol{footnote}}
\footnotetext[1]{\textbf{Funding:} The work of  Yong Li is partially supported by National Natural Science Foundation of China (No. 12071175, 11901056, 11571065),
%Special Funds of Provincial Industrial Innovation of Jilin Province China (No. 2017C028-1),
Jilin Science and Technology Development Program (No. 20190201302JC, 20180101220JC).
The work of Yuecai Han is partially supported by National Natural Science Foundation of China (No. 11871244) and Jilin Science and Technology Development Program (No. 20190201302JC).}

%This work was funded by National Natural Science Foundation of China (No. 11901056,11571065,11871244), Science and Technology Developing Plan of Jilin Province, China (No. 20180101220JC, 20190201302JC), Jilin DRC, China (No. 2017C028-1).}
\footnotetext[2]{School of Mathematics, Jilin University, Changchun 130012, P. R. China. \\ xiaoying21@mails.jlu.edu.cn, hanyc@jlu.edu.cn, liyong@jlu.edu.cn (Corresponding author).}
\footnotetext[3]{School of Mathematics and Statistics, Center for Mathematics and Interdisciplinary Sciences, Northeast Normal University, Changchun 130024, P. R. China.}

\renewcommand{\thefootnote}{\arabic{footnote}}

{\bf Abstract.}
The main focus of this paper is to explore how much similarity between two stochastic differential systems. Motivated by the conjugate theory of stochastic dynamic systems, we study the relationship between two systems by finding homeomorphic mappings $K$. Particularly, we use the minimizer $K^*$ to measure the degree of similarity.
Under appropriate assumptions, we give sufficient and necessary conditions for the existence of the minimizer $K^*$.
The former result can be regarded as a strong law of large numbers, while the latter is a stochastic maximum principle.
%For the later, we theoretically put forward the maximum principle.
Finally, we provide different examples of stochastic systems and an application to stochastic Hartman-Grobman theorem. Thus the results illustrate what is the similarity, extending the conjugacy in stochastic dynamical systems.
%{\bf Mathematics Subject Classifications (2010)}:  37K05, 37K10, 35Q53.\\

\allowdisplaybreaks
%\tableofcontents
\section {Introduction}
\ \ \ \ Many physical systems in nature and society are similar in some characteristics. Although these physical systems may be composed of completely different kinds of microscopic particles or exist on completely different scales, they may follow the same laws and equations at a deeper level. A challenging problem is how to find or extract them quantitatively.
%In the process of exploring the world, people know that microscopic elementary particle are the element that make up everything. From a macro perspective, microscopic particles constitute different physical systems. However, although these physical systems may be composed of completely different kinds of particles or exist on completely different scales, they may follow the same laws and equations at a deeper level. At first glance, they may seem irrelevant, but they may follow the same laws and equations at a deeper level. This means that by studying a specific system, one can understand the information of another system. This is the similarity of the system.
%In 1827, Scottish botanist Robert Brown found that pollen and other suspended small particles in the water kept making irregular Curvilinear motion, called Brownian motion.

From the conjugate theory of dynamical systems, we know that two conjugate systems admit complete same qualitative behaviors.
The earliest work to explore the similarity can be traced back to the analytic linearization theorem proposed by Poincar\'{e} \cite{[P+85]} and Siegel \cite{[S+29]} around the 19th century, who used linear systems to approximate nonlinear systems. After that, the conjugate theory of dynamical systems and some other linearization theorems were established such as Sternberg ($C^k$) linearization theorem \cite{[S+57]} and Hartman-Grobman ($C^0$) linearization theorem \cite{[G+65],[H+60]}. So far, we have known some specific versions of Hartman-Grobman theorem, such as parabolic evolution equations (e.g. scalar reaction-diffusion equations \cite{[L+91]}, Cahn-Hilliard equation and similar phase-field equations \cite{[BL+94]}), the hyperbolic evolution equations (e.g. semilinear \cite{[HP+16]}, nonuniform \cite{[BV+06],[BV+09]}), control systems \cite{[BCP+07]}, and so on. For the smoothness of the conjugacy in the Hartman-Grobman theorem, for example, see \cite{[ZLL+22],[ZLZ+17],[LZZ+20],[DZZ+19],[DZZ+20],[ZZ+16],[ZZ+14],[ZZJ+14]} and the references therein.

In \cite{[WLH+23]}, we study the similarity between two ordinary differential dynamical systems. In this paper, we re-examine the stochastic dynamical systems from the perspective of similarity.
Namely, for given two systems described by stochastic differential equations respectively, what is their
similarity? How can we determine this similarity by means of suitable cost functional? We will touch
these problems.

To be more precise, let $\mathbb{T}=[0,T]$ be a fixed time duration and $\{W_t:t\in\mathbb{T}\}$
%,~\{W_t^1:t\in\mathbb{T}\},~\{W_t^2:t\in\mathbb{T}\}$
be $d$-dimensional standard Brownian motion processes. %, and $\{W_t^1:t\in\mathbb{T}\}$ and $\{W_t^2:t\in\mathbb{T}\}$ be independent of each other.
We consider the following %stochastic control systems, which are
two stochastic differential equations (SDEs): %\footnote{The notation ``$\circ$" in the noise refers to the SDE being understood using It\^{o} integtation. The difference between It\^{o} and Stratonovich will not be an important choice here.}
\begin{equation}\label{2nonlinear}
	\begin{aligned}	
	&\left\{  \begin{aligned}
		dX_t &=f(t,X_t)dt+\sigma(t,X_t) dW_t,\\
		          X_0&=x_0 ,\\ \end{aligned}  \right.\\
	&\left\{  \begin{aligned}
		dY_t &=g(t,Y_t)dt+\varsigma(t,Y_t) dW_t,\\
	             Y_0&=y_0,\\  \end{aligned}  \right.
    \end{aligned}
\end{equation}
where $f(t,X),~g(t,Y) : [0,T]\times \mathbb{R}^n\rightarrow\mathbb{R}^n,~\sigma(t,X),~\varsigma(t,Y) : [0,T]\times \mathbb{R}^n\rightarrow\mathbb{R}^{n\times d}$ are functions, $x_0,y_0\in \mathbb{R}^n$ are initial values. We denote by $X=X(t,x_0),~Y=Y(t,y_0)\in\mathcal{X}$ the solutions of the first and second equations of (\ref{2nonlinear}) respectively, where $\mathcal{X}:=\mathcal{L}^2([0,T]\times\Omega,dt\times\mathbb{P};\R^n)\cap
\mathcal{L}^2([0,T]\times\Omega,dt\times\mathbb{P};\R^{n\times d})$.

Inspired by the conjugate theory of stochastic dynamic systems, we study the similarity of two nonlinear nonautonomous stochastic differential equations (\ref{2nonlinear}). We need to find a homeomorphic (bijective and bicontinuous) mapping $K$ to map one system into another and we define three types of similarity: \\
(i) Complete similarity (Definition \ref{conjugation}):
$$\mathbb{E}||K(X(t,x_0 ))-Y(t,y_0)||_{\mathcal{X}}^2=0,~t\in[0,T].$$
(ii) Asymptotic similarity (Definition \ref{Asymptotic similarity}):
$$\underset{T\rightarrow\infty}{\lim}\mathbb{E}||K(X(T,x_0 ))-Y(T,K(x_0))||_{\mathcal{X}}=0.$$
(iii) Weak similarity (Definition \ref{Cost functional}):
$$\underset{T\rightarrow+\infty}{\lim}J[K]\triangleq
\underset{T\rightarrow\infty}{\lim}\frac{1}{T}\int_0^T\mathbb{E}||K(X(t,x))-Y(t,y)||_{\mathcal{X}}^2dt=0,$$
where $J[K]$ is a cost functional.

It can be seen that the similarity between two random dynamic systems is closely related to the minimizer  $K^*$ ($J[K^*]=\inf J[K]$). % are intricately related.
If we can obtain the sufficient and necessary conditions for the existence of the minimizer $K^*$, then two seemingly unrelated systems can be connected through some similarity transformation, which extends the conjugacy in stochastic dynamic systems. \\

The main results of the present paper read as follows.
For the sufficient existence of the minimizer $K^*$:\\
(Theorem \ref{existence-Ergodicity}). If the Ergodicity assumption (HE) holds, then there exists $K^*$ such that the cost functional $\underset{T\rightarrow+\infty}{\lim}J[K^*]=\text{const},~\mu-a.e.$.\\
(Theorem \ref{existence-Dissipation}). If the Dissipation assumption (HD) holds, then there exists $K^*$ such that the cost functional $\underset{T\rightarrow+\infty}{\lim}J[K^*]=0,~\mathbb{P}-a.e.$.

They can be seen as strong law of large numbers (Theorem \ref{SLLN-Ergodicity} and Theorem \ref{SLLN-Dissipation}):
$$\begin{aligned}
\underset{T\rightarrow\infty}{\lim} J[K]&\triangleq
\underset{T\rightarrow\infty}{\lim}\frac{1}{T}\int_0^T\mathbb{E}||K(X(t,x))-Y(t,y)||_{\mathcal{X}}^2dt\\
&\triangleq\underset{T\rightarrow\infty}{\lim}\frac{1}{T}\int_0^T\phi(X_t,Y_t))dt\\
&=\langle\mu,\phi\rangle=\text{const},~\mathbb{P}-a.e.,
\end{aligned}$$
where %$\Psi(X_t):\mathcal{X}\rightarrow\mathcal{X},\Psi(X_t)\triangleq KX(t,x_0)-Y(t,K(x_0))$, and
$\phi(X_t,Y_t)\triangleq \mathbb{E}||K(X(t,x))-Y(t,y)||_{\mathcal{X}}^2$ is the observable function.

We summarize the above sufficient existence results as Theorem \ref{existence}:
There exists $K^*\in U[0,T]$ such that $J[K^*(\cdot)]=\underset{K(\cdot)\in U[0,T]}{\inf}J[K(\cdot)]$, where $U[0,T]$ is the admissible set that will be defined in the next section.

Then we can define the similarity degree $\rho(J[K])$ to describe the similarity between two stochastic differential systems (\ref{2nonlinear}), see Definition \ref{similarity} for details. \\

For the necessary condition for the existence of the minimizer $K^*$, it can be regarded as a stochastic maximum principle (Theorem \ref{maxsde}):
Let $(X^*(\cdot),Y^*(\cdot),K^*(\cdot))$ be an optimal triple, then there is a quad of processes $(p(\cdot),q(\cdot),r(\cdot),s(\cdot))$ satisfying the first order adjoint equations and with probability 1, one has
\begin{equation*}
\langle H_K(t,X_t^*,Y_t^*,K^*,p_t,q_t,r_t,s_t),K \rangle_{\mathcal{L}^2}\geq0,
\end{equation*}
for a.e. $t\in[0,T],~\forall~K(\cdot)\in U[0,T]$, where the generalized Hamiltonian is defined by\\
$$\begin{aligned}
H&(t,X_t,Y_t,K,p_t,q_t,r_t,s_t)\triangleq\langle p_t,f_t\rangle_{\mathcal{L}^2}+\langle q_t,\sigma_t\rangle_{\mathcal{L}^2}+\langle r_t,g_t\rangle_{\mathcal{L}^2}+\langle s_t,\varsigma_t\rangle_{\mathcal{L}^2}+L(t,X_t,Y_t,K),\\
&(t,X_t,Y_t,K,p_t,q_t,r_t,s_t)\in[0,T]\times\R^n\times\R^n\times U[0,T]\times\R^n\times\R^{n\times d}\times\R^n\times\R^{n\times d},
\end{aligned}$$
and $H_K$ is the partial derivative of Hamiltonian function $H$ with respect to $K$. \\

The remainder of the paper is organized as follows.
In Section \ref{DL}, we recall some definitions and facts concerning stochastic differential equations and nonautonomous systems, and introduce some useful new concepts concerning conjugacy, cost functional, similarity degree and so on.
In Section \ref{Sufficient}, under two different assumptions, i.e., the Ergodic assumption (HE) and the Dissipative assumption (HD), we prove the sufficient existence of the minimizer $K^*$ as a strong law of large numbers.
In Section \ref{MAX}, based on the theory of stochastic optimal control, we give the necessary existence of the minimizer $K^*$, which is a stochastic maximum principle.
In the last section, we illustrate our theoretical results by some examples and apply them to a stochastic Hartman-Grobman theorem.

%in Section \ref{DL}, we put forward some new concepts (conjugacy, similarity, etc.) and the lemmas needed in this paper.

%\newpage
\section{Preliminaries}\label{DL}
\ \ \ \ In this section, we introduce some useful preliminaries, including tightness of measures, nonautonomous dynamical system, skew product flow, tempered random variable, %variational solution,
similarity, conjugacy, cost functional, similarity degree and so on.

\subsection{Definitions}
\ \ \ \ Let $(\mathcal{X},\rho)$ be a complete metric space, $\mathbb{T}=[0,T]$ be a fixed time duration, $\{W_t:t\in\mathbb{T}\}$ be a $d$-dimensional standard Brown motion process, and $(\Omega,\mathcal{F},\mathbb{P})$ be a probability space throughout this paper.

We write $C(\mathbb{T},\mathcal{X})$ to represent the space of all continuous functions
$\varphi : \mathbb{T}\rightarrow \mathcal{X}$ equipped with the distance
$$d(\varphi_1,\varphi_2)=\underset{k=1}{\overset{\infty}{\sum}}\frac{1}{2^k}\frac{d_k(\varphi_1,\varphi_2)}{1+d_k(\varphi_1,\varphi_2)}
$$
when $\mathbb{T}=\R$, where $$d_k(\varphi_1,\varphi_2)=\underset{0\leq t\leq k}{\sup}\rho(\varphi_1(t),\varphi_2(t)).$$
It generates the compact (uniformly convergent on compact time intervals) open topology on $C(\mathbb{T},\mathcal{X})$. The space $C(\mathbb{T},\mathcal{X})$ is a complete metric space.

For SDEs (\ref{2nonlinear}) driven by Brownian motion in $\R^n$, it is well-known that the canonical probability space is $Pr(C(\mathbb{T},\R^n):=(C(\mathbb{T},\R^n),\mathcal{B}(C(\mathbb{T},\R^n)),\mathbb{P}_W)$, where $\mathcal{B}$ is the Borel $\sigma$-algebra, and $\mathbb{P}_W$ is the Wiener measure generated by the Brown motion $W_t$. For brevity, we still use $\mathbb{P}$ to represent $\mathbb{P}_W$.

If the noise term of an SDE is non-degenerate, then it is ergodic and there exists a ergodic invariant measure.
The Wiener shift $\theta_t:\Omega\rightarrow\Omega,~\theta_t W_s\triangleq W_{t+s}-W_t,~\forall~t,s\in\T$ is a measurable, measure-preserving and ergodic dynamical system with invariant measure $\mathbb{P}$, where $\Omega$ is the canonical sample space (Lemma \ref{ergodic dynamical system}).

In addition, we also need to know the convergence (tightness) of the measure. Let $Pr(\mathcal{X})=(\mathcal{X},\mathfrak{B},\mu)$ be a probability space, where $\mathfrak{B}$  is a Borel $\sigma$-algebra.  We write $Pr_2(\mathcal{X})$ to mean the space of probability measures $\mu \in Pr(\mathcal{X})$ such that
$$\int_\mathcal{X} ||z||^2 \mu (dz) < \infty . $$
\begin{definition}\label{Tightness of measures}
\textbf{(Tightness of measures)}
Let $\mathcal{M}\subset Pr(\mathcal{X})$ be a collection of (possibly signed or complex) measures defined on $\mathfrak{B}$. The collection $\mathcal{M}$ is called tight (or sometimes uniformly tight) if, for any $\epsilon >0$, there is a compact subset $K_{\epsilon }\subset\mathcal{X}$ such that, for all measures $\mu \in \mathcal{M}$,
$$\mu(K_\epsilon)>1-\epsilon.$$
\end{definition}
\begin{definition}\label{Precompact in the topology of weak convergence}
\textbf{(Precompact in the topology of weak convergence)}
If any sequence $\{\mu_n\}_{n\in\N}$ in the probability measure collection $\mathcal{M}$ has weakly convergent subsequences, then $\mathcal{M}$ is said to be precompact in the topology of weak convergence.
\end{definition}
\begin{remark}\label{Polish}
(i) If $\mathcal{X}$ is a metrisable compact space, then every collection of (possibly complex) measures on $\mathcal{X}$ is tight. This is not necessarily so for non-metrisable compact spaces.\\
(ii) If $\mathcal{X}$ is a Polish space (separable completely metrisable space), then every probability measure on $\mathcal{X}$ is tight. Furthermore, by Prokhorov's theorem (Lemma \ref{Prokhorov's theorem}), a collection of probability measures on $\mathcal{X}$ is tight if and only if it is precompact in the topology of weak convergence.
\end{remark}

Then we define the nonautonomous dynamical system, skew product flow and shift dynamical system.
\begin{definition}\label{nonautonomous dynamical system}
\textbf{(Nonautonomous dynamical system)} A nonautonomous dynamical system $(\theta, \varphi)$ consists of two ingredients:

(i) A flow $\theta$ on base or parameter space $P$ with time set $\T$, i.e.,\\
(1) $\theta_0(\cdot ) = Id_P,$\\
(2) $\theta_{t+s}(p) = \theta_t(\theta_s(p))$ for all $t, s \in \T$ and $p \in P,$\\
(3) the mapping $(t, p) \mapsto \theta_t(p)$ is continuous.

(ii) A cocycle $\varphi : \T \times P \times S \rightarrow S$, $S$ is the fiber or state space, satisfies the following:\\
(1) $\varphi (0, p, x) = x$ for all $(p, x) \in P \times S,$\\
(2) $\varphi (t + s, p, x) = \varphi (t, \theta_s(p), \varphi (s, p, x))$ for all $s, t \in \T$ and $(p, x) \in P \times S,$\\
(3) the mapping $(t, p, x) \mapsto \varphi (t, p, x)$ is continuous.
\end{definition}

\begin{definition}\label{skew product flow}
\textbf{(Skew product flow)}
Let $(\theta, \varphi)$ be a nonautonomous dynamical system with base space $P$ and state space $S$. The skew product semiflow $\Pi : \T \times P \times S \rightarrow P \times S$ is a semiflow of the form
$$\Pi (t,(p, x)) := (\theta_t(p), \varphi (t, p, x)).$$
\end{definition}

\begin{definition}\label{Shift dynamical system}
\textbf{(Shift dynamical system)}
If the mapping $\pi : \mathbb{T} \times \mathcal{X} \rightarrow \mathcal{X}$ is continuous, and $\pi (0, x) = x$ and $\pi (t + s, x) = \pi (t, \pi (s, x))$ for any $x \in \mathcal{X}$ and $t, s \in \mathbb{T}$,
then we call $(\mathcal{X},\mathbb{T},\pi)$ the shift dynamical system (flow) on $\mathcal{X}$.
\end{definition}

\begin{definition}\label{tempered}
\textbf{(Tempered random variable)}
(i) A random variable $R:\Omega\rightarrow(0,\infty)$ is said to be tempered with respect to a metric dynamical system $\theta$ if
$$\underset {n\rightarrow\pm\infty}{\lim}\frac{1}{n}\log R(\theta^n\omega)=0,~\mathbb{P}-a.s.$$
(ii) $R:\Omega\rightarrow[0,\infty)$ is said to be tempered from above if
$$\underset {n\rightarrow\pm\infty}{\lim}\frac{1}{n}\log^+ R(\theta^n\omega)=0,~\mathbb{P}-a.s,$$
where $\log^+(z)\triangleq\max\{\log(z),0\}$, denoting the non-negative part of the natural logarithm.\\
(iii) $R:\Omega\rightarrow(0,\infty)$ is said to be tempered from below if $\frac{1}{R}$ is tempered from above.
\end{definition}

Let $\mathcal{N}$ to denote the class of $\mathbb{P}$-null sets of $\mathcal{F}$. For each $t\in[0,T]$, $\mathcal{F}_t=\mathcal{N}\vee\sigma\{W_s-W_0:0\leq s\leq t\}$, which is a filtration.

We write $\mathcal{L}^2(\T;\R^n)$ to denote the space of all classes of $\mathcal{F}_t$-measurable stochastic process $\varphi:\T\rightarrow\R^n$ such that
$$\mathbb{E}\int_0^T|\varphi(t)|^2dt<\infty,~dt\times \mathbb{P}~a.e..$$
The space $\mathcal{L}^2(\T;\R^n)$ is a complete metric space. For a given $\varphi(t)\in\mathcal{L}^2(\T;\R^n)$, the forward It\^{o}'s integral $\int_0^{\cdot}\varphi(s)dW_s$ is also in $\mathcal{L}^2(\T;\R^n)$.

We write $\mathcal{L}^1(\T;\R^n)$ to denote the space of all classes of $\mathcal{F}_t$-measurable stochastic process $\varphi:\T\rightarrow\R^n$ such that
$$\mathbb{E}\int_0^T|\varphi(t)|dt<\infty,~dt\times \mathbb{P}~a.e..$$

%\begin{definition}\label{variational solution}
Let $X(t),Y(t),~0\leq t\leq T$ be continuous $\mathcal{F}_t$-adapted processes, and $X,Y\in \mathcal{X}:=\mathcal{L}^2([0,T]\times\Omega,dt\times\mathbb{P};\R^n)\cap
\mathcal{L}^2([0,T]\times\Omega,dt\times\mathbb{P};\R^{n\times d})$.
If for $\mathbb{P}-a.s.~\omega\in\Omega$,
\begin{equation}\label{variational solution}
\begin{aligned}	
    &X(t)=X_0+\int_0^tf(s,X_s)ds+\int_0^t\sigma(s,X_s)dW_s,~0\leq t\leq T,\\
    &Y(t)=Y_0+\int_0^tg(s,Y_s)ds+\int_0^t\varsigma(s,Y_s)dW_s,~0\leq t\leq T,
\end{aligned}
\end{equation}
then we call $(X(t),Y(t))$ a pair of solutions to equation (\ref{2nonlinear}).
%\end{definition}

Let $\Gamma_{x_0}:=\{\rho\in\R^n|\rho=X(t,x_0),~t\in[0,T]\}$ represent the trajectory of $X$, and $\Gamma_{y_0}:=\{\tilde{\rho}\in\R^n|\tilde{\rho}=Y(t,y_0),~t\in[0,T]\}$ represent the trajectory of $Y$.
Let $U=U[0,T]:=\{K\in \mathcal{L}^2(\mathcal{X};\mathcal{X}):\Gamma_{x_0}\rightarrow\Gamma_{y_0},~K\text{ is homeomorphic}\}$ \footnote{\noindent
	Note that $U[0,T]$ is a convex subset of $\mathcal{L}^2(\mathcal{X};\mathcal{X})$.}
%\footnote{\noindent
%	Homeomorphic means bijective and bicontinuous. Note that $\Omega\subset C[\Gamma;\R^n]$. }
denote the admissible set.

%For a convex subset $U\in\R^n$, we define the admissible control set $U[0,T]:=\{\varphi(t)\in\mathcal{L}^2_\mathcal{F}(0,T;\R^n)|\Gamma\rightarrow U,homeomorphic\}$, where $\Gamma:=\{\rho\in\R^n|\rho=X(t,x_0),t\in[0,T]\}$.

\begin{definition}
\textbf{(Complete similarity)}

\textbf{(i) Conjugacy.}\label{conjugation}
Suppose that $X(t),~Y(t)$ satisfy system (\ref{2nonlinear}).
If there admits a homeomorphic mapping $K:\Gamma_{x_0}\rightarrow\Gamma_{y_0}$, such that
\begin{equation}
y_0=K(x_0),~\mathbb{E}||K(X(t,x_0 ))-Y(t,y_0)||_{\mathcal{X}}^2=0,~t\in[0,T],
	%y_0=K(x_0),~\mathbb{E}[K(X(t,x_0 ))]=\mathbb{E}Y(t,y_0),~t\in[0,T],
    %\mathbb{E}y_T=\mathbb{E}[K(x_T)],
\end{equation}
then we call $X(t,x_0 )$ and $Y(t,y_0 )$ conjugate.

\textbf{(ii) Semi-conjugacy.}\label{semi-conjugation}
Suppose that $X(t),Y(t)$ satisfy system (\ref{2nonlinear}).
If there admit two homeomorphic mappings
$K,R:\Gamma_{x_0}\rightarrow\Gamma_{y_0}$, such that
\begin{equation}
	y_0=K(x_0),~\mathbb{E}||R(X(t,x_0 ))]-Y(t,y_0)||_{\mathcal{X}}^2=0,~t\in[0,T],
%y_0=K(x_0),~\mathbb{E}[R(X(t,x_0 ))]=\mathbb{E}Y(t,y_0),~t\in[0,T],
%~\mathbb{E}Y_T=\mathbb{E}[K(X_T)],
\end{equation}
then we call $X(t,x_0 )$ and $Y(t,y_0 )$ semi-conjugate.

If two systems are conjugate (semi-conjugate), we call them completely similar (semi-completely similar).
\end{definition}
\begin{remark}
If $\mathbb{E}||K(X(t,x_0 ))-Y(t,y_0)||_{\mathcal{X}}^2=0$, then
$$\begin{aligned}
0&\leq Var[K(X(t,x_0 ))-Y(t,K(x_0 ))]\\
&=\mathbb{E}||K(X(t,x_0 ))- Y(t,K(x_0))||_{\mathcal{X}}^2-\{\mathbb{E}[K(X(t,x_0 ))-Y(t,K(x_0 ))]\}^2\\
&=0-\{\mathbb{E}[K(X(t,x_0 ))-Y(t,K(x_0 ))]\}^2\leq 0,
\end{aligned}$$
thus $\mathbb{E}||K(X(t,x_0 ))-Y(t,K(x_0 ))||_{\mathcal{X}}=0$.
\end{remark}

Sometimes, there is no such conjugacy or semi-conjugacy between two systems. %in other words, $J[K]\neq0,J[K,R]\neq0$.
So we go back to the second place and find a certain degree of conjugacy, and this is why we propose the following definitions, which are the extensions of Definition \ref{conjugation}: %and Definition \ref{semi-conjugation}:
%definition of cost functional (Definition \ref{Cost functional}).We define the following functionals,
\begin{definition}\label{Asymptotic similarity}
\textbf{(Asymptotic similarity)} Suppose that $X(t)$, $Y(t)$ satisfy (\ref{2nonlinear}). Follow the definitions and notations in Definition \ref{conjugation}, %and Definition \ref{semi-conjugation},
if
$$\underset{T\rightarrow\infty}{\lim}\mathbb{E}||K(X(T,x_0 ))-Y(T,K(x_0))||_{\mathcal{X}}=0,$$
or $$\underset{T\rightarrow\infty}{\lim}\mathbb{E}||R(X(T,x_0 ))-Y(T,K(x_0))||_{\mathcal{X}}=0,$$
then we call $X(t,x_0 )$ and $Y(t,K(x_0) )$ satisfy asymptotic similarity or semi-asymptotic similarity, respectively.
\end{definition}

In the sense of time average, we define the weak similarity with a cost functional. The choice of cost functional is as follows, or it can be in other meaningful forms, such as Onsage-Machlup action functional.
%As a weaker version, we defined the cost functional:
\begin{definition}\label{Cost functional}
\textbf{(Weak similarity)} Suppose that $X(t)$, $Y(t)$ satisfy (\ref{2nonlinear}). Follow the definitions and notations in Definition \ref{conjugation}, %and Definition \ref{semi-conjugation},
and set
%\footnote{$||\cdots||$ represents the norm of vector, such as the $L^1$ norm $||x||=|x_1|+|x_2|+\cdots+|x_n|$ and the $L^2$ norm (Euclidean norm) $||x||=\sqrt{x_1^2+x_2^2+\cdots+x_n^2}$, etc.. They are all equivalent. }
\begin{equation}\label{conjugation functional}
J[K]\triangleq\mathbb{E}\bigg[\dfrac{1}{T} \int_{0}^{T}||K(X(t,x_0))-Y(t,K(x_0))||_{\mathcal{X}}^2 dt\bigg],
%J[K]=\mathbb{E}\bigg[\dfrac{1}{T} \int_{0}^{T}||K(X(t,x_0 ))-Y(t,y_0)||^2 dt+||K(X_T)-Y_T||\bigg],
\end{equation}
\begin{equation}\label{semi-conjugation functional}
J[K,R]\triangleq\mathbb{E}\bigg[\dfrac{1}{T} \int_{0}^{T}||R(X(t,x_0))-Y(t,K(x_0))||_{\mathcal{X}}^2 dt\bigg],
\end{equation}
where $K,R\in \mathcal{L}^2(\mathcal{X};\mathcal{X})$. $\forall~\omega\in\Omega,~K_\omega,R_\omega\in U[0,T]$. $J[K],J[K,R]$ are the lower semi-continuous functions, and obviously, $J[K]\geq0$, $J[K,R]\geq0$. We call them the cost functionals relative to similarity or semi-similarity respectively.

If $\underset{T\rightarrow+\infty}{\lim}J[K]=0~(\underset{T\rightarrow+\infty}{\lim}J[K,R]=0),~\mathbb{P}-a.e.$, we call $X(t,x_0 )$ and $Y(t,K(x_0) )$ satisfy weak similarity with $J[K]$ (semi-weak similarity with $J[K,R]$).
\end{definition}

If the functional $J[K]$ has a minimum and $J[K^*]$ is the minimum, then $X$ and
$Y$ can satisfy the conjugacy to a certain extent. Particularly, if $J[K^*]=0$, then $\mathbb{E}||K^*(X(t,x_0 ))- Y(t,K^*(x_0))||_{\mathcal{X}}^2=0$.
Thus, %$\mathbb{E}[K^*(X(t,x_0 ))]=\mathbb{E}Y(t,K^*(x_0 ))$, i.e.,
$X$ and $Y$ conjugate.

Analogously, if the functional $J[K,R]$ has a minimum and $J[K^*,R^*]$ is the minimum, then $x$
and $y$ can satisfy the semi-conjugacy to a certain extent. Particularly, if
$K=R$, then J$[K,R]$ is reduced to $J[K]$.

\begin{remark}\label{relationship}
From Definitions \ref{conjugation}, \ref{Asymptotic similarity}, \ref{Cost functional}, it can be concluded that
$$\text{Complete similarity}\Rightarrow\text{Asymptotic similarity}\Rightarrow\text{Weak similarity}.$$
\end{remark}

Clearly, the larger the cost functional (Definition \ref{Cost functional}), the smaller the similarity between the two dynamics,
and the ranges of $J[K]$ and $J[K,R]$ are both $[0,+\infty]$. %which is the domain of similarity degree function.

When $J[K]=0~(J[K,R]=0)$, the two systems conjugate (semi-conjugate). In other words, they are completely similar (semi-similar), and the corresponding similarity degree (semi-similarity degree) should be 1. When $J[K]=+\infty~(J[K,R]=+\infty$), the two systems are completely dissimilar, and the corresponding similarity degree (semi-similarity degree) should be 0. Of course, $J[K]$ or $J[K,R]$ are in general finite as $T<\infty$.

It is worth mentioning that these are independent of the selection of the similarity degree function. Based on this, we give the concept of the similarity degree function to quantitatively describe the similarity between two systems \cite{[WLH+23]}.

\begin{definition}\label{similarity}
\textbf{(Similarity degree)} \label{similarity0}
Let the function $\rho(x)$ be continuous and monotonically decreasing in $[0,+\infty]$ such that
$\rho(0)=1$ and $\rho(+\infty)=0$. We call $\rho(x)$ a similarity degree function.

We call $\rho(J[K])$ the similarity degree of systems (\ref{2nonlinear}) with respect to $\rho$ and $\rho(J[K,R])$ the semi-similarity degree of systems (\ref{2nonlinear}) with respect to $\rho$, respectively.
In particular, when conjugating or semi-conjugating, the corresponding similarity and semi-similarity hold with $\rho(J[K^*])=1,~\rho(J[K^*,R^*])=1$.
\end{definition}

For example, $\rho(x)=\frac{\log(1+x)}{x}$. Notice that $\underset{x\rightarrow0^+}{\lim}\frac{\log(1+x)}{x}=1$ and $\underset{x\rightarrow +\infty}{\lim}\frac{\log(1+x)}{x}=0$, then
\begin{equation}
	\rho(J[K])=\left\{ \begin{aligned}	\frac{\log\Big(1+\underset{K\in\{U|K(x_0)=y_0\}}{\inf}~J[K]\Big)}{\underset{K\in\{U|K(x_0)=y_0\}}{\inf}~J[K]}, &\qquad \underset{K\in\{U|K(x_0)=y_0\}}{\inf}~J[K]\neq0,\\
		          1,\qquad\qquad &\qquad \underset{K\in\{U|K(x_0)=y_0\}}{\inf}~J[K]=0, \\ \end{aligned}  \right.
\end{equation}
\begin{equation*}
\rho(J[K,R])=\left\{ \begin{aligned}
		\frac{\log\Big(1+\underset{(K,R)\in\{U\times U|K(x_0)=y_0\}}{\inf}~J[K,R]\Big)}{\underset{(K,R)\in\{U\times U|K(x_0)=y_0\}}{\inf}~J[K,R]}, &\qquad \underset{(K,R)\in\{U\times U|K(x_0)=y_0\}}{\inf}~J[K,R]\neq0,\\
		          1,\qquad\qquad &\qquad \underset{(K,R)\in\{U\times U|K(x_0)=y_0\}}{\inf}~J[K,R]=0. \\ \end{aligned}  \right.
\end{equation*}

The following discussion is mainly for $J[K]$ in the case of a homeomorphism mapping $K$, and it is completely analogous for $J[K,R]$.

The core problem is to find the minimizer $K^*$,
which decides the similarity between two stochastic dynamical systems. If $K^*(\cdot)$ is a constant matrix, then the similarity is the linear similarity. If $K^*(\cdot)$ is orthogonal (metric preserving) or symplectic (differential structure preserving), then the similarity is called the rigid similarity.

\subsection{Settings and main results}\label{Settings and main results}
\ \ \ \ Let $\mathcal{X}:=\mathcal{L}^2([0,T]\times\Omega,dt\times\mathbb{P};\R^n)\cap
\mathcal{L}^2([0,T]\times\Omega,dt\times\mathbb{P};\R^{n\times d})$. We assume that the solution $(X,Y)$ exists and is unique for given mild conditions:\\
\textbf{(H1) (Continuity)}

For all $u,v,w\in \mathcal{X}$, $t\in[0,T],$ the mappings
$$\begin{aligned}
&\R\ni\alpha \mapsto \langle f(t,u+\alpha v),w\rangle_\mathcal{X},\\
&\R\ni\alpha \mapsto \langle g(t,u+\alpha v),w\rangle_\mathcal{X},\\
&\R\ni\alpha \mapsto \langle \sigma(t,u+\alpha v),w\rangle_\mathcal{X},\\
&\R\ni\alpha \mapsto \langle \varsigma(t,u+\alpha v),w\rangle_\mathcal{X}
\end{aligned}$$
are continuous.\\
\textbf{(H2) (Monotonicity)}

For all $u,v\in \mathcal{X}$, $t\in[0,T],$ there exists a constant $c_1$ such that
$$\begin{aligned}
&2\langle f(t,u)-f(t,v),u-v\rangle_\mathcal{X}+||\sigma(t,u)-\sigma(t,v)||^2_{\mathcal{L}^2}\leq c_1||u-v||^2_\mathcal{X},\\
&2\langle g(t,u)-g(t,v),u-v\rangle_\mathcal{X}+||\varsigma(t,u)-\varsigma(t,v)||^2_{\mathcal{L}^2}\leq c_1||u-v||^2_\mathcal{X}.
\end{aligned}$$
\textbf{(H3) (Coerciveness)}

For all $u\in \mathcal{X}$, $t\in[0,T],$ there exist a constant $c_2$ and a positive constant $M_1\in(0,+\infty)$ such that
$$\begin{aligned}
&2\langle f(t,u),u\rangle_\mathcal{X}+||\sigma(t,u)||^2_{\mathcal{L}^2}\leq c_2||u||^2_\mathcal{X}+M_1,\\
&2\langle g(t,u),u\rangle_\mathcal{X}+||\varsigma(t,u)||^2_{\mathcal{L}^2}\leq c_2||u||^2_\mathcal{X}+M_1.
\end{aligned}$$
%(H2) For any $(X,Y)\in\R^n\times\R^n$, $f(\cdot),~g(\cdot)\in\mathcal{L}^2_\mathcal{F}(0,T;\R^n)$, $\sigma(\cdot),~\varsigma(\cdot)\in\mathcal{L}^2_\mathcal{F}(0,T;\R^{n\times d})$.

In order to obtain the continuous dependence of the solution on initial values and coefficients, it is necessary to assume that the following condition holds.\\
\textbf{(H4) (Lipschitz condition)}

For all $u,v\in \mathcal{X}$, $t\in[0,T],$
$f(\cdot),~g(\cdot),~\sigma(\cdot),~\varsigma(\cdot)$ satisfy Lipschitz condition: There exists a positive constant $L>0$ such that
$$\begin{aligned}
&||f(t,u)-f(t,v)||_{\mathcal{L}^2}+||\sigma(t,u)-\sigma(t,v)||_{\mathcal{L}^2}\leq L|u-v|_\mathcal{X},\\
&||g(t,u)-g(t,v)||_{\mathcal{L}^2}+||\varsigma(t,u)-\varsigma(t,v)||_{\mathcal{L}^2}\leq L|u-v|_\mathcal{X}.
\end{aligned}
$$
\begin{lemma}
\textbf{(Well-posedness, \cite{[WU+17]})}
Suppose that (H1)-(H4) hold. Then the solution of equation (\ref{2nonlinear}) exists and is unique. %within the meaning of Definition \ref{variational solution}.
\end{lemma}

To prove the (sufficient) existence of the minimizer $K^*$ of functional $J[K]$, where $K\in U[0,T]$, we first consider $KX(t,x_0)-Y(t,K(x_0))$.
According to (\ref{variational solution}), for $\mathbb{P}-a.s.~\omega\in\Omega$,
\begin{equation}\label{Psi}
KX_t-Y_t=\int_0^t[Kf(s,X_s)-g(s,Y_s)]ds+\int_0^t[K\sigma(s,X_s)-\varsigma(s,Y_s)]dW_s,~0\leq t\leq T.
\end{equation}

Then, under some appropriate assumptions (ergodicity or dissipation), we have the (sufficient) existence of the minimizer $K^*$ (Theorem \ref{existence}) in Section \ref{Sufficient}, which can be seen as a strong law of large numbers (SLLN, \cite{[LL+22]}).\\
\textbf{(HE) (Ergodicity)}

For all $x,y\in \mathcal{X}$, $t\in[0,T]$, $X(t,x)$ and $Y(t,y)$ are ergodic in probability space $(\mathcal{X},\mathfrak{B},\mu)$.\\
%The mapping $\Psi:\mathcal{X}\rightarrow\mathcal{X},\Psi(X_t)\triangleq KX(t,x_0)-Y(t,K(x_0))$ is ergodic in probability space $(\mathcal{X},\mathfrak{B},\mu)$.\\
%\textbf{(H) We assume that (\ref{2nonlinear}) is ergodic and has a unique ergodic invariant measure $\mu^*$.}
\textbf{(HD) (Dissipation)}

For all $X_t,Y_t\in \mathcal{X}$, $t\in[0,T],$ there exist positive constants $\alpha_1>0$ such that
$$\begin{aligned}
2\langle \frac{\partial K}{\partial X^T}f(t,X_t)-g(t,Y_t),KX_t-Y_t\rangle_{\mathcal{X}}+||\frac{\partial K}{\partial X^T}\sigma(t,X_t)-\varsigma(t,Y_t)||_{\mathcal{L}^2}^2
&\leq -\alpha_1||KX_t-Y_t||_{\mathcal{X}}^2.
%\sigma(t,X_t)\in\mathcal{L}^2(\mathbb{T};\R^{n\times d}),&~\varsigma(t,Y_t)\in\mathcal{L}^2(\mathbb{T};\R^{n\times d}).
%\int_0^T\mathbb{E}\big| ||K\sigma(t,X_t)||_{\mathcal{L}^2}^2-||\varsigma(t,Y_t)||^2_{\mathcal{L}^2}\big|dt&\leq M.
%,\\ ||K^*\sigma(t,X_t)-\varsigma(t,Y_t)||_{\mathcal{L}^2}&\leq L_\sigma||K^*X_t-Y_t||_{\mathcal{X}}.
\end{aligned}$$

Under the (HE) assumption, we obtain the first result (Theorem \ref{existence-Ergodicity}):

For $x,y\in\mathcal{X}$, suppose that $X(t,x),Y(t,y)$ are the solutions of equation (\ref{2nonlinear}) and (H1)-(H4) hold. If (HE) holds, then there exists $K^*$ such that $\underset{T\rightarrow+\infty}{\lim}J[K^*]=\text{const},~\mu-a.e.$.

Under the (HD) assumption, we obtain the second result (Theorem \ref{existence-Dissipation}):

Suppose that $X_t,Y_t\in\mathcal{X}$ are the %variational
solutions of equation (\ref{2nonlinear}) and (H1)-(H4) hold. If (HD) holds, then there exists $K^*$ such that the cost functional $\underset{T\rightarrow+\infty}{\lim}J[K^*]=0,~\mathbb{P}-a.e.$.

They can be written in the form of SLLN (Theorem \ref{SLLN-Ergodicity} and Theorem \ref{SLLN-Dissipation}, respectively):
$$\begin{aligned}
\underset{T\rightarrow\infty}{\lim} J[K^*]&\triangleq
\underset{T\rightarrow\infty}{\lim}\frac{1}{T}\int_0^T\mathbb{E}||K^*(X(t,x))-Y(t,y)||_{\mathcal{X}}^2dt\\
&\triangleq\underset{T\rightarrow\infty}{\lim}\frac{1}{T}\int_0^T\phi(X_t,Y_t))dt\\
&=\langle\mu,\phi\rangle=\text{const},~\mathbb{P}-a.e.,
\end{aligned}$$
where $\phi(X_t,Y_t)\triangleq \mathbb{E}||K^*(X(t,x))-Y(t,y)||_{\mathcal{X}}^2$ is the observable function.\\
%The first result is Theorem \ref{existence-Ergodicity}. It can be seen as an SLLN (Theorem ):
%$$\underset{T\rightarrow\infty}{\lim}\frac{1}{T}\int_0^T\phi(\Psi(X_t))dt=0,~\mathbb{P}-a.e..$$

%The second result is Theorem \ref{exponential decay estimate}. We can write Theorem \ref{exponential decay estimate} in the form of SLLN (Theorem). To prove Theorem \ref{exponential decay estimate}, we need the following It\^{o} formula and Gronwall lemma.

Our third result is the necessary condition for existence of the minimizer $K^*(t)$, which is a stochastic maximum principle (Theorem \ref{maxsde}):
%Then we consider how to determine in Section \ref{MAX},

Let $(X^*(\cdot),Y^*(\cdot),K^*(\cdot))$ be an optimal triple of the control problem. Then there is a quad of processes $(p(\cdot),q(\cdot),r(\cdot),s(\cdot))$ satisfying the first order adjoint equations and with probability 1, one has
\begin{equation}\label{geq0}
\langle H_K(t,X_t^*,Y_t^*,K^*,p_t,q_t,r_t,s_t),K \rangle_{\mathcal{L}^2}\geq0,
\end{equation}
for a.e. $t\in[0,T],~\forall~K(\cdot)\in U[0,T]$, where\\
$$\begin{aligned}
H&(t,X_t,Y_t,K,p_t,q_t,r_t,s_t)\triangleq\langle p_t,f_t\rangle_{\mathcal{L}^2}+\langle q_t,\sigma_t\rangle_{\mathcal{L}^2}+\langle r_t,g_t\rangle_{\mathcal{L}^2}+\langle s_t,\varsigma_t\rangle_{\mathcal{L}^2}+L(t,X_t,Y_t,K),\\
&(t,X_t,Y_t,K,p_t,q_t,r_t,s_t)\in[0,T]\times\R^n\times\R^n\times U[0,T]\times\R^n\times\R^{n\times d}\times\R^n\times\R^{n\times d}.
\end{aligned}$$

During the proof process, we need to further assume that the following conditions hold:\\
\textbf{(H1') (Differentiability)}

For all $u,v,w\in \mathcal{X}$, $t\in[0,T],$ $f(t,u),~g(t,u),~\sigma(t,u),~\varsigma(t,u)$ are continuously differentiable with respect to $u$, i.e., the mappings
$$\begin{aligned}
&\R\ni\alpha \mapsto \langle f_X(t,u+\alpha v),w\rangle_\mathcal{X},\\
&\R\ni\alpha \mapsto \langle g_Y(t,u+\alpha v),w\rangle_\mathcal{X},\\
&\R\ni\alpha \mapsto \langle \sigma_X(t,u+\alpha v),w\rangle_\mathcal{X},\\
&\R\ni\alpha \mapsto \langle \varsigma_Y(t,u+\alpha v),w\rangle_\mathcal{X}
\end{aligned}$$
are continuous.\\
\textbf{(H3') (Boundness of derivative)}

For all $u\in \mathcal{X}$, $t\in[0,T],$ there exist a positive constant $M_2\in(0,+\infty)$ such that
$$\begin{aligned}
&2\langle f_X(t,u),u\rangle_\mathcal{X}+||\sigma_X(t,u)||^2_{\mathcal{L}^2}\leq M_2,\\
&2\langle g_Y(t,u),u\rangle_\mathcal{X}+||\varsigma_Y(t,u)||^2_{\mathcal{L}^2}\leq M_2.
\end{aligned}$$

%$f(t,\cdot),~g(t,\cdot),~\sigma(t,\cdot),~\varsigma(t,\cdot)$ are continuously differentiable. with respect to $X,Y$, and the first derivatives are bounded.\\
%(H4) The linear growth condition: $\exists~C>0$ s.t. $\forall~t\in [0,T]$,
%$$
%\begin{aligned}
%|f(t,X(t))|+|\sigma(t,X(t))|&\leq C(1+|X(t)|),\\
%|g(t,Y(t))|+|\varsigma(t,Y(t))|&\leq C(1+|Y(t)|).
%\end{aligned}
%$$

In Section \ref{Applications}, we provide three different examples and summarize the corresponding similarity results here:\\
(i) A steady linear system and its output system are completely similar (conjugate); \\
(ii) Two steady linear systems can satisfy complete similarity, asymptotic similarity or that the similarity degree is 1 as $T\rightarrow\infty$, under conditions from strong to weak; \\
(iii) A nonlinear system and its linearization system are completely similar (conjugate) near the fixed point (Lemma \ref{stochastic Hartman-Grobman theorem} and Theorem \ref{stochastic Hartman-Grobman theorem1}).

%After the above preparation, we can obtain the necessary conditions for existence (Theorem \ref{maxsde}, the stochastic maximum principle) in Section \ref{MAX}.

\section{The (sufficient) existence of $K^*$}\label{Sufficient}
\ \ \ \ We want to know under what conditions the functional $J[K]$ reaches the minimum, that is, find the minimizer $K^*$ to make the functional reach the minimum.
As mentioned in Section \ref{Settings and main results}, in this section, we will obtain the (sufficient) existence of the minimizer $K^*$ under either the ergodic or dissipative assumptions, which can be regarded as an SLLN.

\subsection{The SLLN under Ergodicity assumption}
\ \ \ \ Denote by $\mathbb{F}:=(f,\sigma,g,\varsigma)$. We write $F^\tau(t,x)=F(t +\tau,x),F\in\{f,\sigma,g,\varsigma\}$ for all $(t, x) \in \T \times \mathcal{X}$, the hull $\mathcal{H}(\mathbb{F})$ is the closure of $\{\mathbb{F}^\tau=(f^\tau,\sigma^\tau,g^\tau,\varsigma^\tau),\tau\in\T\}$.
$(\mathcal{H}(\mathbb{F}),\T, \theta)$ is a shift dynamical system, where $\theta: \T \times \mathcal{H}(\mathbb{F}) \rightarrow \mathcal{H}(\mathbb{F}),(\tau ,\mathbb{F} ) \mapsto \mathbb{F}^\tau$.

Define $P_\mathbb{F}(t,dy):=\mathbb{P} \circ (\theta(t))^{-1}(dy)$. Then we can associate a mapping $P^*(t,\mathbb{F}, \cdot):Pr(\mathcal{X}) \rightarrow Pr(\mathcal{X})$ defined by
$$P^*(t,\mathbb{F}, \mu )(B):=\int_\mathcal{X} P_\mathbb{F} (t, B)\mu (dx)$$
for all $\mu \in Pr(\mathcal{X})$ and $B \in \mathcal{B}(\mathcal{X})$.
Then $P^*$ is a cocycle on $(\mathcal{H}(\mathbb{F}),\T, \theta)$ with fiber $Pr_2(\mathcal{X})$.
\begin{proposition}
(\cite{[CL+23]})
The mapping given by
$$\Pi : \T\times \mathcal{H}(\mathbb{F}) \times Pr_2(\mathcal{X}) \rightarrow \mathcal{H}(\mathbb{F}) \times Pr_2(\mathcal{X}),$$
$$\Pi (t,(\tilde{\mathbb{F}}, \mu )) := \Bigl(\theta_t(\tilde{\mathbb{F}}), P^*(t,\tilde{\mathbb{F}}, \mu ))$$
is a continuous skew product flow, where $\tilde{\mathbb{F}}\in\mathcal{H}(\mathbb{F})$ and $\mu\in Pr_2(\mathcal{X})$.
\end{proposition}

%Firstly, we note that since equation (\ref{2nonlinear}) is driven by Brownian motion, ergodicity naturally holds.
For ergodic measures, we have the following Birkhoff ergodic theorem:
\begin{lemma}\label{Birkhoff ergodic theorem}
\textbf{(Birkhoff ergodic theorem)}
Let $(\mathcal{X},\mathfrak{B},\mu)$ be a probability space and $T:\mathcal{X}\rightarrow\mathcal{X}$ be an ergodic transformation. Then for any $\phi\in \mathcal{L}^1(\mu)$,
$$\underset{m\rightarrow\infty}{\lim}\frac{1}{m}\underset{i=0}{\overset{m-1}{\sum}}\phi(T^i x)=\int\phi d\mu,~\mu-a.e.,~x\in\mathcal{X}.$$
\end{lemma}
Recall the Ergodicity assumption (HE):\\
\textbf{(HE) (Ergodicity)}

For all $x,y\in \mathcal{X}$, $t\in[0,T]$, $X(t,x)$ and $Y(t,y)$ are ergodic in probability space $(\mathcal{X},\mathfrak{B},\mu)$.\\

Based on the Ergodicity assumption (HE) and the Birkhoff ergodic theorem (Lemma \ref{Birkhoff ergodic theorem}), we can obtain the following theorem.
\begin{theorem}\label{existence-Ergodicity}
For $x,y\in\mathcal{X}$, suppose that $X(t,x),Y(t,y)$ are the solutions of equation (\ref{2nonlinear}) and (H1)-(H4) hold. If (HE) holds, then there exists $K^*$ such that %the cost functional (\ref{conjugation functional})
$\underset{T\rightarrow+\infty}{\lim}J[K^*]=\text{const},~\mu-a.e.$.
%$K\in U[0,T]$ satisfies $K(x_0)=y_0$, then $\underset{T\rightarrow+\infty}{\lim}J[K]=0$.
\end{theorem}
\begin{proof} We divide the proof into 3 steps.\\

\textbf{Step 1.} Verify the observable function $\phi(X_t,Y_t)\in\mathcal{L}^1(\mu)$.

For $x,y\in\mathcal{X}$, define the observable function $\phi(X_t,Y_t)\triangleq \mathbb{E}||K(X(t,x))-Y(t,y)||_{\mathcal{X}}^2$. %, where $\Psi(X_t):\mathcal{X}\rightarrow\mathcal{X},\Psi(X_t)\triangleq KX(t,x_0)-Y(t,K(x_0))$.
According to the definitions and assumptions in Section \ref{DL}, $X_t,Y_t\in\mathcal{X}:=\mathcal{L}^2([0,T]\times\Omega,dt\times\mathbb{P};\R^n)\cap
\mathcal{L}^2([0,T]\times\Omega,dt\times\mathbb{P};\R^{n\times d})$, $K\in U[0,T]:=\{K\in \mathcal{L}^2(\mathcal{X};\mathcal{X}):\Gamma_{x_0}\rightarrow\Gamma_{y_0},~K\text{ is homeomorphic}\}$. Therefore, $\phi(X_t,Y_t)\in\mathcal{L}^1(\mu)$.\\

\textbf{Step 2.} Skew product flow properties.

Due to $K$ being a homeomorphism, under the Ergodicity assumption, $KX_t-Y_t$ is a ergodic skew product flow $\Pi (t,(\tilde{\mathbb{F}}, \mu ))$ with the ergodic, measure-preserving transformation $\Psi$ in $Pr(\mathcal{X})$.

For each fixed $T>0$, divide $[0, T]$ into $m$ segments, each segment is $h_m=\frac{T}{m}$ long. Let $\tau_i=ih_m,~i=0,1,\cdots,m$. For $x,y\in\mathcal{X}$, let
\begin{equation}\begin{aligned}
\xi_0&=KX_0-Y_0\triangleq\Psi^0(x,y),\\
\xi_i&=KX_{\tau_i}-Y_{\tau_i}\triangleq\Psi^i(x,y),i=1,2,\cdots,m.
\end{aligned}
\end{equation}
%\begin{equation}\begin{aligned}
%\xi_0&=\Pi (0,(\tilde{\mathbb{F}}, \mu ))\triangleq\Psi^0,\\
%\xi_i&=\Pi (\tau_i,(\tilde{\mathbb{F}}, \mu ))\triangleq\Psi^i,i=1,2,\cdots,m.
%\end{aligned}
%\end{equation}
Thereupon, we get $(m+1)$ discrete points: $$(\tau_0,\xi_0),(\tau_1,\xi_1),\cdots,(\tau_m,\xi_m).$$
From the cost functional (\ref{conjugation functional}) and Fubini theorem, we have
\begin{equation}\label{tau_i}
\begin{aligned}
\underset{T\rightarrow+\infty}{\lim}J[K^*]&=\underset{T\rightarrow+\infty}{\lim}
\mathbb{E}\bigg[\dfrac{1}{T} \int_{0}^{T}||K^*(X(t,x))-Y(t,y)||_{\mathcal{X}}^2 dt\bigg]\\
&=\underset{T\rightarrow+\infty}{\lim}\frac{1}{T}\int_0^T\phi(X_t,Y_t)dt\\
&=\underset{T\rightarrow+\infty}{\lim}\frac{1}{T}\underset{i=0}{\overset{m-1}{\sum}}\int_{\tau_i}^{\tau_{i+1}}\phi(X_t,Y_t)dt\\
&=\underset{T\rightarrow+\infty}{\lim}\frac{1}{T}\underset{i=0}{\overset{m-1}{\sum}}\frac{T}{m}\phi(\xi_i)\\
&=\underset{m\rightarrow\infty}{\lim}\frac{1}{m}\underset{i=0}{\overset{m-1}{\sum}}\phi(\Psi^i(x,y)).
\end{aligned}
\end{equation}

\textbf{Step 3.} Birkhoff ergodic theorem.

As can be seen from the previous text, $Pr(\mathcal{X})=(\mathcal{X},\mathfrak{B},\mu)$ is a probability space. %and $\Psi:\mathcal{X}\rightarrow\mathcal{X}$ is an ergodic, measure-preserving transformation in $Pr(\mathcal{X})$ (HE).
According to Birkhoff ergodic theorem (Lemma \ref{Birkhoff ergodic theorem}),
\begin{equation}\label{Birkhoff}
\underset{m\rightarrow\infty}{\lim}\frac{1}{m}\underset{i=0}{\overset{m-1}{\sum}}\phi(\Psi^i (x,y))=\int\phi d\mu,~\mu-a.e..%,~x,y\in\mathcal{X}.
\end{equation}
Substituting the above equation into (\ref{tau_i}), we obtain
$$\begin{aligned}
\underset{T\rightarrow+\infty}{\lim}J[K^*]&=\underset{m\rightarrow\infty}{\lim}\frac{1}{m}\underset{i=0}{\overset{m-1}{\sum}}\phi(\Psi^i (x,y))\\
&=\int\phi(x,y) d\mu\\
&=\langle\mu,\phi\rangle\\
%&=\mathbb{E}_{\mu}||K^*x-y||_{\mathcal{X}}^2\\%,~\mu-a.e.,~x,y\in\mathcal{X}.
&=\iint_{\mathcal{X}\times\mathcal{X}}|K^*x-y|^2dxdy\\
&=\text{const},~\mu-a.e..
\end{aligned}$$
%i.e.,
%$$\underset{T\rightarrow+\infty}{\lim}J[K^*]=0,~\mu-a.e..$$
\end{proof}

From another perspective, we can write Theorem \ref{existence-Ergodicity} in the form of SLLN (Theorem \ref{SLLN-Ergodicity}). To prove Theorem \ref{SLLN-Ergodicity}, we need the following Skorokhod's representation theorem.
\begin{lemma}\label{Skorokhod's representation theorem}
\textbf{(Skorokhod's representation theorem)}
Let $\{\mu_n\}_{n\in\N}$ be a sequence of probability measures on a metric space $\mathcal{X}$ such that $\mu_n$ converges weakly to some probability measure $\mu$ on $\mathcal{X}$ as $n\rightarrow\infty$. Suppose also that the support of $\mu$ is separable. Then there exist $\mathcal{X}$-valued random variables $\phi_n,\phi$ defined on a common probability space $(\Omega ,{\mathcal {F}},\mathbb{P})$ such that the law of $\phi_n$ is $\mu_n$ for all $n$, the law of $\phi$ is $\mu$, and $\phi_n\overset{n\rightarrow\infty}{\longrightarrow}\phi$ $\mathbb{P}$-almost surely.
\end{lemma}

\begin{theorem}\label{SLLN-Ergodicity}
\textbf{(The SLLN under Ergodicity assumption)}
Suppose that all the conditions of Theorem \ref{existence-Ergodicity} hold. Then
$$\underset{T\rightarrow\infty}{\lim}\frac{1}{T}\int_0^T\phi(X_t,Y_t)dt=\langle\mu,\phi\rangle,~\mathbb{P}-a.e.,$$
where $\phi(X_t,Y_t)\triangleq \mathbb{E}||K^*(X(t,x))-Y(t,y)||_{\mathcal{X}}^2$ is the observable function.
\end{theorem}
\begin{proof}
Comparing the results of Theorem \ref{existence-Ergodicity} with the proof process, we only need to prove the convergence of the measure additionally.

Obviously, $\mathcal{X}$ is a Polish space, then for every probability measure on $\mathcal{X}$ is tight (Remark \ref{Polish} $(ii)$). By Prokhorov's theorem (Lemma \ref{Prokhorov's theorem}), the collection of probability measures on $\mathcal{X}$ is precompact in the topology of weak convergence. According to Definition \ref{Precompact in the topology of weak convergence}, for all sequence $\{\mu_n\}_{n\in\N}$ in the probability measure collection $\mathcal{M}$ has weakly convergent subsequences on $\mathcal{X}$.

For a fixed $T>0$, there exists an invariant measure $\mu_T$. As $T\rightarrow\infty$, we can obtain a subsequence, denoted as $\{\mu_n\}_{n\in\N}$. Suppose that $\mu_n$ converges weakly to some probability measure $\mu_0$ on $\mathcal{L}^2(Pr(\mathcal{X}))$, i.e., $$\mu_n\overset{n\rightarrow\infty}{\rightharpoonup}\mu_0,$$ and the support of $\mu_0$ is separable.

According to Skorokhod's representation theorem (Lemma \ref{Skorokhod's representation theorem}), there exist $\mathcal{X}$-valued random variables $\phi_n,\phi$ defined on a common probability space $(\Omega,\mathcal{F},\mathbf{P})$ such that the law of $\phi_n$ is $\mu_n$ for all $n$, the law of $\phi$ is $\mu_0$, and $$\phi_n\overset{n\rightarrow\infty}{\rightarrow}\phi,~\mathbf{P}-a.s..$$

For the Wiener measure $\mathbb{P}$, we can take the same operation as $\mu$. Let $\mathbb{P}_n\in Pr(C(\mathbb{T},\R^n)$ be the law of $\phi_n$ for all $n$, and $\mathbb{P}_n$ converge weakly to some probability measure $\mathbb{P}_0$ on $\mathcal{L}^2(Pr(C(\mathbb{T},\R^n))$, i.e., $$\mathbb{P}_n\overset{n\rightarrow\infty}{\rightharpoonup}\mathbb{P}_0,$$
where $\mathbb{P}_0$ is the law of $\phi$.
Then we can replace $\mu$ in Theorem \ref{existence-Ergodicity} with $\mathbb{P}$.
\end{proof}

\subsection{The SLLN under Dissipation assumption}
\ \ \ \ Recall the Dissipation assumption (HD):\\
\textbf{(HD) (Dissipation)}

For all $X_t,Y_t\in \mathcal{X}$, $t\in[0,T],$ there exist positive constants $\alpha_1>0$ such that
$$\begin{aligned}
2\langle \frac{\partial K}{\partial X^T}f(t,X_t)-g(t,Y_t),KX_t-Y_t\rangle_{\mathcal{X}}+||\frac{\partial K}{\partial X^T}\sigma(t,X_t)-\varsigma(t,Y_t)||_{\mathcal{L}^2}^2
&\leq -\alpha_1||KX_t-Y_t||_{\mathcal{X}}^2.%,\\
%\sigma(t,X_t)\in\mathcal{L}^2(\mathbb{T};\R^{n\times d}),&~\varsigma(t,Y_t)\in\mathcal{L}^2(\mathbb{T};\R^{n\times d}).
\end{aligned}$$
\begin{theorem}\label{existence-Dissipation}
Suppose that $X_t,Y_t\in\mathcal{X}$ are the %variational
solutions of equation (\ref{2nonlinear}) and (H1)-(H4) hold. If (HD) holds, then there exists $K^*$ such that the cost functional $\underset{T\rightarrow+\infty}{\lim}J[K^*]=0,~\mathbb{P}-a.e.$.
%Suppose that $K(x_0)=y_0$ and (HK) hold, then $J[K]=0$.
\end{theorem}
\begin{proof}
Let $\varphi(t)=KX_t-Y_t$,
$$\begin{aligned}
d\varphi(t)&=\big(\frac{\partial K}{\partial X^T}f(t,X_t)-g(t,Y_t)\big)dt+\big(\frac{\partial K}{\partial X^T}\sigma(t,X_t)-\varsigma(t,Y_t)\big)dW_t\\
&\triangleq F(t,\varphi)dt+G(t,\varphi)dW_t.
\end{aligned}$$
Let $C^{2,1}(\mathbb{R}^n\times\mathbb{T};~\mathbb{R}^+)$ denote the family of all nonnegative functions $V(\varphi,t)$ which are continuously twice differentiable in $\varphi$ and once differentiable in $t$.
Let $V(\varphi,t)\in C^{2,1}(\mathbb{R}^n_0\times\mathbb{T};~\mathbb{R}^+)$ be a positive definite Lyapunov function, where $\mathbb{R}^n_0=\mathbb{R}^n-\{\mathbf{0}\}$ and $V(\mathbf{0},\cdot)=0$. Define an operator $LV(\varphi,t):\mathbb{R}^n\times\mathbb{T}\rightarrow\mathbb{R}$ by
$$LV(\varphi,t)=V_t(\varphi,t)+V_\varphi(\varphi,t)F(\varphi,t)+\frac{1}{2}
Tr[G^T(\varphi,t)V_{\varphi\varphi}(\varphi,t)G(\varphi,t)],$$
where
$$V_t(\varphi,t)=\frac{\partial V(\varphi,t)}{\partial t},
~V_\varphi(\varphi,t)=\frac{\partial V(\varphi,t)}{\partial \varphi},
~V_{\varphi\varphi}(\varphi,t)=\bigg(\frac{\partial^2V(\varphi,t)}{\partial \varphi^2}\bigg)_{n\times n}.$$

Obviously, $V(\varphi,t)=||\varphi||^2$ is a positive definite Lyapunov function, under the Dissipation assumption (HD), we can calculate that
$$\begin{aligned}
LV(\varphi,t)&=0+2\varphi^TF(t,\varphi)+\frac{1}{2}
Tr[G^T(\varphi,t)2G(\varphi,t)]\\
&=2\langle \frac{\partial K}{\partial X^T}f(t,X_t)-g(t,Y_t),KX_t-Y_t\rangle_{\mathcal{X}}+||\frac{\partial K}{\partial X^T}\sigma(t,X_t)-\varsigma(t,Y_t)||_{\mathcal{L}^2}^2\\
&\leq -\alpha_1||KX_t-Y_t||_{\mathcal{X}}^2\\
&=-\alpha_1||\varphi||^2.
\end{aligned}$$
According to Lemma \ref{Lyapunov function theorem},
$$\underset{t\rightarrow\infty}{\lim\sup}\frac{1}{t}\log(\mathbb{E}||\varphi(t)||^2)\leq-\alpha_1.$$
It follows that
%$$\mathbb{E}(||K^*X_t-Y_t||^2_{\mathcal{X}})=0,~t\in[0,T].$$
%i.e., $X(t,x_0 )$ and $Y(t,K^*(x_0))$ are conjugate (Definition \ref{conjugation}).
%\begin{equation}\label{asymptotic exponential decay}
%\begin{aligned}
%0&\leq\underset{t\rightarrow+\infty}{\lim}\mathbb{E}(||K^*X_t-Y_t||^2_{\mathcal{X}})\\
%&\leq \underset{t\rightarrow+\infty}{\lim}\alpha(t)e^{-\alpha_1t}\\
%&=0,
%\end{aligned}
%\end{equation}
%where $\underset{t\rightarrow+\infty}{\lim}\frac{\alpha(t)}{t}=0$. Hence,
$$\underset{t\rightarrow+\infty}{\lim}\mathbb{E}(||K^*X_t-Y_t||^2_{\mathcal{X}})=0,$$
i.e., $X(t,x_0)$ and $Y(t,K^*(x_0))$ satisfy asymptotic similarity (Definition \ref{Asymptotic similarity}).

Further,
%$$\begin{aligned}
%\mathbb{E}\bigg[\dfrac{1}{T} \int_{0}^{T}||K^*(X(t,x_0))-Y(t,K^*(x_0))||^2 dt\bigg]
%&\leq\dfrac{M_T}{T} \int_{0}^{T}e^{-Ct}dt\\
%&=\dfrac{M_T}{CT} \bigg(1-e^{-CT}\bigg).
%\end{aligned}$$
%Let $T\rightarrow+\infty$,
$$\underset{T\rightarrow+\infty}{\lim}\mathbb{E}\bigg[\dfrac{1}{T} \int_{0}^{T}||K^*(X(t,x_0))-Y(t,K^*(x_0))||^2 dt\bigg]=0,~\mathbb{P}-a.e.,$$
i.e., the cost functional (\ref{conjugation functional}) satisfies
\begin{equation}\label{J[K]=0}
\underset{T\rightarrow+\infty}{\lim}J[K^*]=0,~\mathbb{P}-a.e..
\end{equation}
\end{proof}
%\begin{remark}
%In fact, under the of Dissipation assumption (HD), $X(t,x_0 )$ and $Y(t,K(x_0) )$ satisfy asymptotic similarity. From (\ref{asymptotic exponential decay}), we know that $\mathbb{E}(||K^*X_T-Y_T||^2_{\mathcal{X}})\leq M_Te^{-CT}$, where $C=\alpha_1>0$. Hence
%$$\begin{aligned}0\leq\underset{T\rightarrow\infty}{\lim}&\mathbb{E}(||K^*X_T-Y_T||^2_{\mathcal{X}})\leq \underset{T\rightarrow\infty}{\lim}M_Te^{-CT}=0,\\
%\underset{T\rightarrow\infty}{\lim}&\mathbb{E}(||K^*X_T-Y_T||_{\mathcal{X}})=0.
%\end{aligned}$$
%According to Definition \ref{Asymptotic similarity}, $X(t,x_0 )$ and $Y(t,K(x_0) )$ satisfy asymptotic similarity.
%\end{remark}

We can write Theorem \ref{existence-Dissipation} in the form of SLLN (Theorem \ref{SLLN-Dissipation}): %which is completely equivalent to Theorem \ref{exponential decay estimate}.
\begin{theorem}\label{SLLN-Dissipation}
\textbf{(The SLLN under Dissipation assumption)}
Suppose that all the conditions of Theorem \ref{existence-Dissipation} hold. Then
$$\underset{T\rightarrow\infty}{\lim}\frac{1}{T}\int_0^T\phi(X_t,Y_t)dt=0,~\mathbb{P}-a.e.,$$
where $\phi(X_t,Y_t)\triangleq \mathbb{E}||K^*(X(t,x_0))-Y(t,K^*(x_0))||_{\mathcal{X}}^2$ is the observable function.
\end{theorem}

If $\underset{T\rightarrow+\infty}{\lim}J[K]=\text{const}$, then $J[K]$ has a minimum value on the finite interval $[0, T]$, and the $K$ corresponding to the minimum value is the $K^*$ we want. We give the following theorem as a summary of this section.

\begin{theorem}\label{existence}
\textbf{(Existence of $K^*$)}
Suppose that $X_t,Y_t\in\mathcal{X}$ are the %variational
solutions of equation (\ref{2nonlinear}) and (H1)-(H4) hold. Also suppose that (HE) or (HD) holds. Then there exists $K^*\in U[0,T]$ such that $J[K^*(\cdot)]=\underset{K(\cdot)\in U[0,T]}{\inf}J[K(\cdot)]$.
\end{theorem}
\begin{proof}
The proof is direct.
According to Theorem \ref{existence-Ergodicity} and Theorem \ref{existence-Dissipation}, $\underset{T\rightarrow+\infty}{\lim}J[K^*]=\text{const}$. Then we make a truncation of the time interval. Due to the lower semi-continuity of $J[K]$, it has a minimum value on the finite interval $[0, T]$, i.e.,
$$
J[K^*(\cdot)]=\underset{K(\cdot)\in U[0,T]}{\inf}J[K(\cdot)].
$$
\end{proof}

\section{A stochastic maximum principle}\label{MAX}
\ \ \ \ In Section \ref{Sufficient}, we have proven the sufficient existence of the minimizer $K^*$. Naturally, one should think about the similarity between two stochastic systems. Motivated by this question, we study the necessary condition of the minimizer $K^*$, which is a stochastic maximum principle.\\

%Firstly, we make a truncation of the time interval. We do not consider $t>T$, but consider the terminal term $\Phi(X_T,Y_T)$.
As mentioned in Section \ref{Settings and main results},
let $L(t,X_t,Y_t,K)=\dfrac{1}{T} ||KX(t,x_0 )-Y(t,y_0)||^2$, $h(X_T,Y_T)=||K(X_T)-Y_T||$ the cost functional (\ref{conjugation functional}) is changed into
\begin{equation}\label{cost functional}
\tilde{J}[K(\cdot)]=\mathbb{E}\bigg[\int_{0}^{T}L(t,X_t,Y_t,K) dt+h(X_T,Y_T)\bigg].
\end{equation}
It can be stated as minimizing the cost functional (\ref{cost functional}) by the optimal control $K^*$, i.e.,
$$
\tilde{J}[K^*(\cdot)]=\underset{K(\cdot)\in U[0,T]}{\inf}\tilde{J}[K(\cdot)].
$$
The corresponding $(X^*(\cdot),Y^*(\cdot))$ is called an optimal state process. We also call $(X^*(\cdot),Y^*(\cdot),K^*(\cdot))$ an optimal triple.\\

Now, we seek the necessary conditions that the optimal control should meet. Let $K(\cdot)\in U[0,T]$ make $K(\cdot)+K^*(\cdot)$ be an admissible control. It can be seen from the convexity of $U$ that for any $0\leq\epsilon\leq1$, $K^*(\cdot)+\epsilon K(\cdot)$ is also an admissible control. Let $K^\epsilon(\cdot)= K^*(\cdot)+\epsilon K(\cdot)$, and the corresponding solution to that state equation (\ref{2nonlinear}) is $(X^\epsilon(\cdot),Y^\epsilon(\cdot))$. For simplicity, we rewrite $\varphi(\cdot,X^*(\cdot),Y^*(\cdot),K^*(\cdot))$ as $\varphi^*(\cdot)$ and $\varphi(\cdot,X^\epsilon(\cdot),Y^\epsilon(\cdot),K^\epsilon(\cdot))$ as $\varphi^\epsilon(\cdot)$.
The following stochastic maximum principle is the main result of this section.

\begin{theorem}\label{maxsde}
\textbf{ (Maximum principle)} Let $(X^*(\cdot),Y^*(\cdot),K^*(\cdot))$ be an optimal triple of the control problem. Then there is a quad of processes $(p(\cdot),q(\cdot),r(\cdot),s(\cdot))$ satisfying the first order adjoint equations (\ref{adjoint equations}) and with probability 1, one has
\begin{equation}\label{geq0}
\langle H_K(t,X_t^*,Y_t^*,K^*,p_t,q_t,r_t,s_t),K \rangle_{\mathcal{L}^2}\geq0,
\end{equation}
for a.e. $t\in[0,T],~\forall~K(\cdot)\in U[0,T]$, where\\
$$\begin{aligned}
H&(t,X_t,Y_t,K,p_t,q_t,r_t,s_t)\triangleq\langle p_t,f_t\rangle_{\mathcal{L}^2}+\langle q_t,\sigma_t\rangle_{\mathcal{L}^2}+\langle r_t,g_t\rangle_{\mathcal{L}^2}+\langle s_t,\varsigma_t\rangle_{\mathcal{L}^2}+L(t,X_t,Y_t,K),\\
&(t,X_t,Y_t,K,p_t,q_t,r_t,s_t)\in[0,T]\times\R^n\times\R^n\times U[0,T]\times\R^n\times\R^{n\times d}\times\R^n\times\R^{n\times d},
\end{aligned}$$
and $H_K$ is the partial derivative of Hamiltonian function $H$ with respect to $K$.
\end{theorem}
\begin{proof}
Its proof can be divided into the following steps:\\

\textbf{Step 1. Variational inequality.}

Due to $\underset{T\rightarrow+\infty}{\lim}J[K]=\text{constant}$, we make a truncation of the time interval. We do not consider $t>T$, but consider the terminal term $h(X_T,Y_T)=||K(X_T)-Y_T||$.
Then the cost functional (\ref{conjugation functional}) is changed into
\begin{equation}
\tilde{J}[K(\cdot)]=\mathbb{E}\bigg[\int_{0}^{T}\dfrac{1}{T} ||KX(t,x_0 )-Y(t,y_0)||^2dt+h(X_T,Y_T)\bigg].
\end{equation}
Let $L(t,X_t,Y_t,K)=\dfrac{1}{T} ||KX(t,x_0 )-Y(t,y_0)||^2$, then the following variational inequality holds:
\begin{equation*}
\begin{aligned}
\frac{d}{d\epsilon}\tilde{J}[K^*(\cdot)+\epsilon K(\cdot)]|_{\epsilon=0}&=\underset{\epsilon\rightarrow0}{\lim}\frac{1}{\epsilon}\{\tilde{J}[K^\epsilon(\cdot)]-\tilde{J}[K^*(\cdot)]\}\\
&=\underset{\epsilon\rightarrow0}{\lim}\frac{1}{\epsilon}\mathbb{E}\int_0^T(L^\epsilon(t)-L^*(t))dt
+\underset{\epsilon\rightarrow0}{\lim}\frac{1}{\epsilon}\mathbb{E}[h^\epsilon(\cdot)-h^*(\cdot)]\\
&\geq0.
\end{aligned}
\end{equation*}
For the right side of the above equation, we have
$$\begin{aligned}
\frac{1}{\epsilon}\mathbb{E}\int_0^T(L^\epsilon(t)-L^*(t))dt=&\frac{1}{\epsilon}\mathbb{E}\int_0^T(L^\epsilon(t)-L(t,X^*(t),Y^\epsilon(t),K^\epsilon))dt\\
&+\frac{1}{\epsilon}\mathbb{E}\int_0^T(L(t,X^*(t),Y^\epsilon(t),K^\epsilon)-L(t,X^*(t),Y^*(t),K^\epsilon))dt\\
&+\frac{1}{\epsilon}\mathbb{E}\int_0^T(L(t,X^*(t),Y^*(t),K^\epsilon)-L^*(t))dt,\\
\frac{1}{\epsilon}\mathbb{E}[h^\epsilon(\cdot)-h^*(\cdot)]=&\frac{1}{\epsilon}\mathbb{E}[h^\epsilon(\cdot)-h(X^*_T,Y^\epsilon_T)]
+\frac{1}{\epsilon}\mathbb{E}[h(X^*_T,Y^\epsilon_T)-h^*(\cdot)].
\end{aligned}$$

In order to study the two terms of the above variational inequality, we make some preparations in the next two steps.\\

\textbf{Step 2. Variational equations.}% and Proposition \ref{variational}.}

Introduce the variational equations:
\begin{equation}\label{variational equations}
\begin{aligned}	
&\left\{  \begin{aligned}
d\hat{X_t} &=f_X^*(t)\hat{X_t}dt+\sigma_X^*(t)\hat{X_t}dW_t,\\
\hat{X_0}&=\mathbf{0}, \\ \end{aligned}  \right.\\
&\left\{  \begin{aligned}
d\hat{Y_t} &=g_Y^*(t)\hat{Y_t}dt+\varsigma_Y^*(t)\hat{Y_t}dW_t,\\
\hat{Y_0}&=\mathbf{0}.\\  \end{aligned}  \right.
\end{aligned}
\end{equation}
\begin{proposition}\label{variational}
Let $(X_t^\epsilon,Y_t^\epsilon),(X_t^*,Y_t^*)$ satisfy (\ref{2nonlinear}) and $(\hat{X_t},\hat{Y_t})$ satisfy (\ref{variational equations}). Then
$$
\mathbb{E}||X_t^\epsilon-X_t^*-\epsilon\hat{X_t}||_{\mathcal{X}}^2\leq C_\epsilon\epsilon^2,~ \mathbb{E}||Y_t^\epsilon-Y_t^*-\epsilon\hat{Y_t}||_{\mathcal{X}}^2\leq C_\epsilon\epsilon^2,
$$
where $C_\epsilon\rightarrow0$ as $\epsilon\rightarrow0$.
\end{proposition}
\begin{proof} Let $\xi^\epsilon_t=X_t^\epsilon-X_t^*-\epsilon\hat{X_t}$ and $\eta^\epsilon_t=Y_t^\epsilon-Y_t^*-\epsilon\hat{Y_t}$. Substituting stochastic differential equations (\ref{2nonlinear}) and (\ref{variational equations}), we can get that $\xi^\epsilon_t$ satisfies the following SDE:
$$
d\xi^\epsilon_t=(f^\epsilon(t)-f^*(t)-\epsilon f_X^*(t)\hat{X_t})dt+(\sigma^\epsilon(t)-\sigma^*(t)-\epsilon \sigma_X^*(t)\hat{X_t})dW_t.
$$
By using It\^{o}'s formula (Lemma \ref{Ito formula}), we can obtain
$$\begin{aligned}
d||\xi^\epsilon_t||_{\mathcal{X}}^2=&2\langle f^\epsilon(t)-f^*(t)-\epsilon f_X^*(t)\hat{X_t},\xi^\epsilon_t\rangle_{\mathcal{X}} dt\\
&+2\langle\xi^\epsilon_t,(\sigma^\epsilon(t)-\sigma^*(t)-\epsilon \sigma_X^*(t)\hat{X_t})dW_t\rangle_{\mathcal{X}}\\
&+||\sigma^\epsilon(t)-\sigma^*(t)-\epsilon \sigma_X^*(t)\hat{X_t}||_{\mathcal{L}^2}^2dt.
\end{aligned}
$$

Notice that $\xi^\epsilon_0=\mathbf{0}$, $||\xi^\epsilon_0||_{\mathcal{X}}^2=0$. Then
$$\begin{aligned}
||\xi^\epsilon_t||_{\mathcal{X}}^2=&\int_0^t2\langle f^\epsilon(s)-f^*(s)-\epsilon f_X^*(s)\hat{X_s},\xi^\epsilon_s\rangle_{\mathcal{X}} ds\\
&+\int_0^t2\langle\xi^\epsilon_s,(\sigma^\epsilon(s)-\sigma^*(s)-\epsilon \sigma_X^*(s)\hat{X_s})dW_s\rangle_{\mathcal{X}}\\
&+\int_0^t||\sigma^\epsilon(s)-\sigma^*(s)-\epsilon \sigma_X^*(s)\hat{X_s}||_{\mathcal{L}^2}^2ds,\\
\mathbb{E}||\xi^\epsilon_t||_{\mathcal{X}}^2=&\mathbb{E}\int_0^t2\langle f^\epsilon(s)-f^*(s)-\epsilon f_X^*(s)\hat{X_s},\xi^\epsilon_s\rangle_{\mathcal{X}} ds\\
&+\mathbb{E}\int_0^t||\sigma^\epsilon(s)-\sigma^*(s)-\epsilon \sigma_X^*(s)\hat{X_s}||_{\mathcal{L}^2}^2ds.
\end{aligned}
$$
Hence from assumptions (H1'), (H3'), (H4) and the Young inequality (Lemma \ref{Young inequality}) , we deduce
$$
\begin{aligned}
\mathbb{E}||\xi^\epsilon_t||_{\mathcal{X}}^2\leq&\mathbb{E}\int_0^t||\xi^\epsilon_s||_{\mathcal{X}}^2
+||f^\epsilon(s)-f^*(s)-\epsilon f_X^*(s)\hat{X_s}||_{\mathcal{L}^2}^2ds\\
&+\int_0^t||\sigma^\epsilon(s)-\sigma^*(s)-\epsilon \sigma_X^*(s)\hat{X_s}||_{\mathcal{L}^2}^2ds\\
\leq&\mathbb{E}\int_0^t||\xi^\epsilon_s||_{\mathcal{X}}^2
+\epsilon_0||\xi^\epsilon_s||_{\mathcal{X}}^2+C_\epsilon\epsilon^2ds\\
\leq&\mathbb{E}\int_0^t(1+\epsilon_0)||\xi^\epsilon_s||_{\mathcal{X}}^2ds+C_\epsilon\epsilon^2T\\
\triangleq& C\mathbb{E}\int_0^t|\xi^\epsilon_s|^2ds+C_\epsilon\epsilon^2,
\end{aligned}
$$
where $\epsilon_0,C,C_\epsilon$ are positive constants, and $C_\epsilon\rightarrow0$ as $\epsilon\rightarrow0$.

According to Gronwall's inequality, we obtain
$$
\mathbb{E}|\xi^\epsilon_t|^2\leq C_\epsilon\epsilon^2e^{Ct}\leq C_\epsilon\epsilon^2e^{CT}\leq C_\epsilon\epsilon^2.
$$
Analogically, we can obtain the result that $\mathbb{E}|\eta^\epsilon_t|^2\leq C_\epsilon\epsilon^2$.
\end{proof}

\textbf{Step 3. Adjoint equations.}% and Proposition \ref{pxry}.}

Consider the adjoint equations, which are two backward stochastic differential equations (BSDEs):
\begin{equation}\label{adjoint equations}
\begin{aligned}	
&\left\{  \begin{aligned}
-dp_t &=[f_X^*(t)p_t+\sigma_X^*(t)q_t+L_X^*(t)]dt-q_tdW_t,\\
p_T&=h_X^*(X_T,Y_T), \\ \end{aligned}  \right.\\
&\left\{  \begin{aligned}
-dr_t &=[g_Y^*(t)r_t+\varsigma_Y^*(t)s_t+L_Y^*(t)]dt-s_tdW_t,\\
r_T&=h_Y^*(X_T,Y_T).\\ \end{aligned}  \right.
\end{aligned}
\end{equation}

\begin{proposition}\label{pxry}
Let $(\hat{X_t},\hat{Y_t})$ satisfy (\ref{variational equations}) and $(p_t,q_t,r_t,s_t)$ satisfy (\ref{adjoint equations}). Then
$$\begin{aligned}
\mathbb{E}\langle p_T,\hat{X_T}\rangle&=\mathbb{E}\int_0^T-\langle L_X^*(s),\hat{X_s} \rangle_{\mathcal{X}} ds,\\
\mathbb{E}\langle r_T,\hat{Y_T}\rangle&=\mathbb{E}\int_0^T-\langle L_Y^*(s),\hat{Y_s} \rangle_{\mathcal{X}} ds.
\end{aligned}
$$
\end{proposition}
\begin{proof}
By using It\^{o}'s formula (\ref{Ito formula}), we can obtain
$$\begin{aligned}
d\langle p_t,\hat{X_t}\rangle=&\langle p_t,f_X^*(t)\hat{X_t}dt+\sigma_X^*(t)\hat{X_t}dW_t\rangle_{\mathcal{X}}\\
&-\langle [f_X^*(t)p_t+\sigma_X^*(t)q_t+L_X^*(t)]dt-q_tdW_t,\hat{X_t}\rangle_{\mathcal{X}}\\
&+\langle q_t,\sigma_X^*(t)\hat{X_t} \rangle_{\mathcal{X}} dt,\\
d\langle r_t,\hat{Y_t}\rangle=&\langle r_t,g_Y^*(t)\hat{Y_t}dt+\varsigma_Y^*(t)\hat{Y_t}dW_t\rangle_{\mathcal{X}}\\
&-\langle [g_Y^*(t)r_t+\varsigma_Y^*(t)s_t+L_Y^*(t)]dt-s_tdW_t,\hat{Y_t}\rangle_{\mathcal{X}}\\
&+\langle s_t,\varsigma_Y^*(t)\hat{Y_t} \rangle_{\mathcal{X}} dt.
\end{aligned}
$$
Leveraging the fact that $\hat{X_0}=0$, we obtain
$$\begin{aligned}
\langle p_t,\hat{X_t}\rangle=&\int_0^t\langle p_s,f_X^*(s)\hat{X_s}ds+\sigma_X^*(s)\hat{X_s}dW_s\rangle_{\mathcal{X}}\\
&-\int_0^t\langle [f_X^*(s)p_s+\sigma_X^*(s)q_s+L_X^*(s)]ds-q_sdW_s,\hat{X_s}\rangle_{\mathcal{X}}\\
&+\int_0^t\langle q_s,\sigma_X^*(s)\hat{X_s} \rangle_{\mathcal{X}} ds\\
=&\int_0^t\langle p_s,\sigma_X^*(s)\hat{X_s}dW_s\rangle_{\mathcal{X}}-\int_0^t\langle L_X^*(s)ds-q_sdW_s,\hat{X_s}\rangle_{\mathcal{X}},\\
\langle r_t,\hat{Y_t}\rangle=&\int_0^t\langle r_s,g_Y^*(s)\hat{Y_s}ds+\varsigma_Y^*(s)\hat{Y_s}dW_s\rangle_{\mathcal{X}}\\
&-\int_0^t\langle [g_Y^*(s)r_s+\varsigma_Y^*(s)s_s+L_Y^*(s)]ds-s_sdW_s,\hat{Y_s}\rangle_{\mathcal{X}}\\
&+\int_0^t\langle s_s,\varsigma_Y^*(s)\hat{Y_s} \rangle_{\mathcal{X}} ds\\
=&\int_0^t\langle r_s,\varsigma_Y^*(s)\hat{Y_s}dW_s\rangle_{\mathcal{X}}-\int_0^t\langle L_Y^*(s)ds-s_sdW_s,\hat{Y_s}\rangle_{\mathcal{X}}.
\end{aligned}
$$
Let $t=T$, and then take the expectation. We get the result of the proposition.
\end{proof}

\textbf{Step 4. Stochastic maximum principle.}

Now we can prove the stochastic maximum principle (the necessary conditions for the existence of the minimizer).

By Proposition \ref{variational}, we have
$$\begin{aligned}
\frac{1}{\epsilon}\mathbb{E}\int_0^T(L^\epsilon(t)-L^*(t))dt&\overset{\epsilon\rightarrow0}{\longrightarrow}\mathbb{E}\int_0^T
L^*_X(t)\hat{X_t}+L^*_Y(t)\hat{Y_t}+L^*_K(t)Kdt,\\
\frac{1}{\epsilon}\mathbb{E}[h^\epsilon(\cdot)-h^*(\cdot)]&\overset{\epsilon\rightarrow0}{\longrightarrow}\mathbb{E}
\bigg[h_X^*(X_T,Y_T)\hat{X_T}+h_Y^*(X_T,Y_T)\hat{Y_T}\bigg].
\end{aligned}$$
Hence we obtain the following variational inequality:
\begin{equation}\begin{aligned}
\frac{d}{d\epsilon}J[K^\epsilon(\cdot)]|_{\epsilon=0}=&\mathbb{E}\bigg\{\int_0^TL^*_X(t)\hat{X_t}
+L^*_Y(t)\hat{Y_t}+L^*_K(t)Kdt\\
&+\bigg[h_X^*(X_T,Y_T)\hat{X_T}+h_Y^*(X_T,Y_T)\hat{Y_T}\bigg]\bigg\}\geq0.
\end{aligned}
\end{equation}

According to Proposition \ref{pxry}, $\frac{d}{d\epsilon}J[K^\epsilon(\cdot)]|_{\epsilon=0}=
\mathbb{E}\{\int_0^T\langle L^*_K(t),K\rangle dt\}\geq0$. Defining the generalized Hamiltonian by
$$\begin{aligned}
H&(t,X_t,Y_t,K,p_t,q_t,r_t,s_t)\triangleq\langle p_t,f_t\rangle_{\mathcal{L}^2}+\langle q_t,\sigma_t\rangle_{\mathcal{L}^2}+\langle r_t,g_t\rangle_{\mathcal{L}^2}+\langle s_t,\varsigma_t\rangle_{\mathcal{L}^2}+L(t,X_t,Y_t,K),\\
&(t,X_t,Y_t,K,p_t,q_t,r_t,s_t)\in[0,T]\times\R^n\times\R^n\times U[0,T]\times\R^n\times\R^{n\times d}\times\R^n\times\R^{n\times d},
\end{aligned}$$
we have
$$
\mathbb{E}\{\int_0^T \langle H_K(t,X_t^*,Y_t^*,K^*,p_t,q_t,r_t,s_t),K\rangle_{\mathcal{L}^2} dt\}\geq0.
$$
Therefore,
$$
\mathbb{E}[\mathbf{1}_A \langle H_K(t,X_t^*,Y_t^*,K^*,p_t,q_t,r_t,s_t),K\rangle_{\mathcal{L}^2}]\geq0,
~\forall~A\in\mathcal{F}_t.
$$
This completes the proof.
\end{proof}

\begin{remark}
Further, if we assume that $U[0,T]=\mathcal{L}^2([0,T];\R^n)$, then from (\ref{geq0}), we have that $H_K(t,X_t^*,Y_t^*,K_t^*,p_t,q_t,r_t,s_t)=0$, $\mathbb{P}$-a.s.. Moreover, we assume that $H_{KK}<0$. Then it is well-known that $K^*(\cdot)$ is the unique solution to maximum $H_K(t,X^*,Y^*,K^*,p,q,r,s)$. From the implicit function theorem, we also know that $K^*(\cdot)$ is uniquely represented as the function of $(X^*(\cdot),Y^*(\cdot),p(\cdot),q(\cdot),r(\cdot),s(\cdot))$.
\end{remark}

\subsection{Applications of maximum principle}\label{Applications}
\ \ \ \ In this section, we solve $K^*$ which satisfies the stochastic maximum principle (Theorem \ref{maxsde}). We consider the following equation:
\begin{equation}\label{linear and nonlinear1}
	\begin{aligned}	
	&\left\{  \begin{aligned}
		dX_t &=f(X_t)dt+\sigma(X_t)dW_t,\\
		          X_0&=x_0 ,\\ \end{aligned}  \right.\\
	&\left\{  \begin{aligned}
		dY_t &=g(Y_t)dt+\varsigma(Y_t)dW_t,\\
	             Y_0&=y_0,\\  \end{aligned}  \right.
    \end{aligned}
\end{equation}
where $\sigma(X_t)\sigma^T(X_t),~t\in[0,T]$ and $\varsigma(Y_t)\varsigma^T(Y_t),~t\in[0,T]$ are invertible (the determinant are not zero).
\begin{proposition}\label{HJB-K}
Suppose that $X_t,Y_t\in\mathcal{X}$ are the %variational
solutions of equation (\ref{linear and nonlinear1}) and (H1)-(H4) hold. If there exists $K^*$ such that $$\mathbb{E}||K^*X_t-Y_t||^2=0,~t\in[0,T],$$
then $K^*$ satisfies the following condition:
$$
\frac{\partial K^*}{\partial X^T}=\bigg[\varsigma(Y)\sigma^T(X)-f(X)(K^*X-Y)^T\bigg]
\bigg[\sigma(X)\sigma^T(X)\bigg]^{-1}.
$$

\end{proposition}
\begin{proof}
Let $\psi_t^K=KX_t-Y_t$. Then
\begin{equation}\label{linear-nonlinear-Psi}
\left\{\begin{aligned}
d\psi_t^K=&\bigg(\frac{\partial K}{\partial X^T}f(X_t)-g(Y_t)\bigg)dt+\bigg(\frac{\partial K}{\partial X^T}\sigma(X_t)-\varsigma(Y_t)\bigg)dW_t\\
\triangleq&F(\psi_t^K,t)dt+G(\psi_t^K,t)dW_t,\\
\psi_0^K=&\psi_0.
\end{aligned}\right.
\end{equation}
The cost functional corresponding to (\ref{cost functional}) is
$$\begin{aligned}
J[K]&=\mathbb{E}\bigg[\int_0^{T}\frac{1}{T}||KX_t-Y_t||^2dt+||KX_{T}-Y_{T}||\bigg]\\
&\triangleq \mathbb{E}\bigg[\int_0^{T} L(\psi_t^K)dt+h(\psi_{T)}^K)\bigg].
\end{aligned}$$

Define $V(\psi,t)=\underset{K}{\inf}~\mathbb{E}\big[\int_t^{T} L(\psi_s^K)ds+h(\psi_{T}^K)\big],~t\in[0,T]$. Using the Bellman dynamic programming methods, $V(\psi,t)$ satisfies the following Hamilton-Jacobi-Bellman (HJB) equation:
\begin{equation}\label{HJB-d}
\left\{\begin{aligned}
&V_t+\underset{K}{\inf}\bigg\{V_\psi F(\psi,t)+\frac{1}{2}Tr\bigg[G^T(\psi,t)V_{\psi\psi}G(\psi,t)\bigg]
+L(\psi)\bigg\}=0,\\
&V(\psi,T))=h(\psi_{T}^K),
\end{aligned}\right.
\end{equation}
where $V_t:=\frac{\partial V}{\partial t}(\psi,t),~V_\psi:=\frac{\partial V}{\partial \psi}(\psi,t),~V_{\psi\psi}:=\frac{\partial^2 V}{\partial \psi^2}(\psi,t)$.
Let $\psi$ be a solution path, we have $V_\psi=2\psi^T,~V_{\psi\psi}=2I$.

Since the minimizer $K^*$ exists, according to the stochastic maximum principle (Theorem \ref{maxsde}), from (\ref{HJB-d}) we can obtain
$$\begin{aligned}
&V_\psi F(\psi^{K^*},t)+\frac{1}{2}Tr\bigg[G^T(\psi^{K^*},t)V_{\psi\psi}G(\psi^{K^*},t)\bigg]\\
=&V_\psi\big(\frac{\partial K^*}{\partial X^T}f(X)-g(Y)\big)+
\frac{1}{2}\big(\frac{\partial K^*}{\partial X^T}\sigma(X)-\varsigma(Y)\big)^TV_{\psi\psi}\big(\frac{\partial K^*}{\partial X^T}\sigma(X)-\varsigma(Y)\big)\\
=&0.
\end{aligned}$$
Substituting $V_\psi=2\psi^T$ and $V_{\psi\psi}=2I$ into the above equation, we get
\begin{equation}\label{K^* equation}
2\langle \frac{\partial K^*}{\partial X^T}f(X)-g(Y),K^*X-Y\rangle
+||\frac{\partial K^*}{\partial X^T}\sigma(X)-\varsigma(Y)||^2=0.
\end{equation}

Using It\^{o}'s formula (Lemma \ref{Ito formula}) to $||K^*X_t-Y_t||^2$, we have
$$\begin{aligned}
\mathbb{E}||K^*X_t-Y_t||^2=&\mathbb{E}\int_0^t2\langle \frac{\partial K^*}{\partial X^T}f(X_s)-g(Y_s),K^*X_s-Y_s\rangle \\
&+||\frac{\partial K^*}{\partial X^T}\sigma(X)-\varsigma(Y)||^2ds.
\end{aligned}$$
Setting $\psi=\psi_s^{K^*}$ and substituting (\ref{K^* equation}) into the above equation,  we obtain
$$\mathbb{E}||K^*X_t-Y_t||^2=0,~t\in[0,T],$$
i.e., $X(t,x_0 )$ and $Y(t,K^*(x_0))$ are conjugate (Definition \ref{conjugation}) in $[0,T]$.

Since the noise term is non-degenerate, i.e., $\sigma(X_t)\sigma^T(X_t)$ is invertible,
%$\underset{l=1}{\overset{d}{\sum}}f_l(X)f_l^T(X)$ is invertible
we can solve from (\ref{K^* equation}) that the minimizer $K^*$ satisfies the following equation:
$$f(X)(K^*X-Y)^T+\bigg(\frac{\partial K^*}{\partial X^T}\sigma(X)-\varsigma(Y)\bigg)\sigma^T(X)=\mathbf{0}_{n\times n},$$
i.e.,
\begin{equation}\label{maxsde K^*d}
\frac{\partial K^*}{\partial X^T}=\bigg[\varsigma(Y)\sigma^T(X)-f(X)(K^*X-Y)^T\bigg]
\bigg[\sigma(X)\sigma^T(X)\bigg]^{-1}.
\end{equation}
\end{proof}

To express $K^*$ in (\ref{maxsde K^*d}) more explicitly, we consider the case that $n=d=1$.
Then $K^*$ follows the stochastic maximum principle (Theorem \ref{maxsde}):
\begin{equation}\label{maxsde K^*}
\frac{\partial K^*}{\partial x}=\frac{V_\psi f(x)-V_{\psi\psi}\sigma(x)\varsigma(y)}{-V_{\psi\psi}\sigma^2(x)}
=\frac{\psi f(x)-\sigma(x)\varsigma(y)}{-\sigma^2(x)}.
\end{equation}
Substituting (\ref{maxsde K^*}) back to (\ref{HJB-d}) yields
$$2[f(x)\sigma(x)\varsigma(y)-\sigma^2(x)g(y)]V_{\psi\psi}\cdot V_\psi-f^2(x)V_\psi^2=0,$$
i.e., $V_\psi=2\psi=0$ or
$$2[f(x)\sigma(x)\varsigma(y)-\sigma^2(x)g(y)]V_{\psi\psi}-f^2(x)V_\psi=0.$$
Obviously, the case $V_\psi=0$ can deduce that $V(\psi,t)\equiv0$ has a naturally trivial solution $\psi=0$. We can solve for $V(\psi,t)$ from the other equation and then substitute it back into (\ref{maxsde K^*}) to get $K^*$:
\begin{equation}\label{maxsde K^*2}
\frac{\partial K^*}{\partial x}=\frac{-f(x)\varsigma(y)+2\sigma(x)g(y)}{f(x)\sigma(x)}.
\end{equation}

\section{Applications}\label{Applications}
\ \ \ \ In this section, we provide some examples. They are a steady linear system and its output system, two steady linear systems, and a linear system and a nonlinear system. %We summarize the similarity relationship between them in Remark \ref{summarize SDE eg}.\\

Follow the notations and assumptions in Section \ref{DL}, $\mathcal{X}:=\mathcal{L}^2([0,T]\times\Omega,dt\times\mathbb{P};\R^n)\cap
\mathcal{L}^2([0,T]\times\Omega,dt\times\mathbb{P};\R^{n\times d})$, and $K\in U:=\{K\in \mathcal{L}^2(\mathcal{X};\mathcal{X}):\Gamma_{x_0}\rightarrow\Gamma_{y_0},~K\text{ is homeomorphic}\}$.\\
%\textbf{(HE) (Ergodicity)}
%
%For all $x,y\in \mathcal{X}$, $t\in[0,T]$, $X(t,x)$ and $Y(t,y)$ are ergodic in probability space $(\mathcal{X},\mathfrak{B},\mu)$.\\
%\textbf{(HD) (Dissipation)}
%
%For all $X_t,Y_t\in \mathcal{X}$, $t\in[0,T],$ there exist positive constants $\alpha_1>0$ such that
%$$\begin{aligned}
%2\langle \frac{\partial K}{\partial X}f(t,X_t)-g(t,Y_t),KX_t-Y_t\rangle_{\mathcal{X}}+||\frac{\partial K}{\partial X}\sigma(t,X_t)-\varsigma(t,Y_t)||_{\mathcal{L}^2}^2
%&\leq -\alpha_1||KX_t-Y_t||_{\mathcal{X}}^2.
%\end{aligned}$$

Before showing the examples, we first recall Oseledets theorem (Multiplicative ergodic theorem), which provides the theoretical background for computation of Lyapunov exponents.
\begin{lemma}\label{Oseledets theorem}
\textbf{(Oseledets theorem or Multiplicative ergodic theorem, \cite{[D+15]})}
Let $\Phi(t,\omega)$ be a linear random dynamical system (RDS, i.e., a linear cocycle) in $\R^n$, for $t\in[0,+\infty)$ on a probability space $(\Omega,\mathcal{F},\mathbb{P})$, over a measurable driving flow $\theta_t$. Assume that the following integrability conditions are satisfied:
$$\underset{0\leq t\leq1}{\sup}\log^+||\Phi(t,\omega)||\in\mathcal{L}^1(\Omega).$$
Then there exists an invariant set $\tilde{\Omega}\in\mathcal{F}$ (i.e., $\theta_t^{-1}\tilde{\Omega}=\tilde{\Omega}$) of the full probability measure, such that for every $\omega\in\tilde{\Omega}$ the following statements hold:\\
(i) $\underset{t\rightarrow\infty}{\lim}[\Phi(t,\omega)^T\Phi(t,\omega)]^{\frac{1}{2t}}=\bar{\Phi}(\omega)$ exists, and $\bar{\Phi}(\omega)$ is a non-negative $n\times n$ matrix.\\
(ii) The matrix $\bar{\Phi}(\omega)$ has distinct eigenvalues $e^{\lambda_{p(\omega)}(\omega)}<\cdots<e^{\lambda_1(\omega)}$ with corresponding eigenspaces $E_{p(\omega)}(\omega),\cdots,E_1(\omega)$ of dimensions $d_i(\omega)=\dim E_i(\omega),i=1,\cdots,p(\omega)$, and these eigenspaces are such that $\R^n=E_1(\omega)\oplus\cdots\oplus E_{p(\omega)}(\omega)$. Then
$$p(\theta\omega)=p(\omega),\lambda_i(\theta\omega)=\lambda_i(\omega),d_i(\theta\omega)=d_i(\omega).$$
(iii) If $\theta$ is ergodic, i.e., every measurable invariant set of $\theta$ have probability 0 or 1, then the functions $p(\omega)$, $\lambda_i(\omega)$, and $d_i(\omega)$ are constants on $\tilde{\Omega}$.\\
(iv) Each $E_i(\omega)$ is invariant for the linear random dynamical system $\Phi$ in the following sense: $\Phi(t,\omega)E_i(\omega)=E_i(\theta_t\omega)$.\\
(v) $\underset{t\rightarrow\infty}{\lim}\frac{1}{t}\log\frac{||\Phi(t,\omega)x||}{||x||}=\lambda_i(\omega)$ if and only if $x\in E_i(\omega)\backslash\{0\}$, $i=1,2,\cdots,p(\omega)$.\\
(vi) The functions $\omega\mapsto p(\omega)\in\{1,\cdots,d\},\omega\mapsto \lambda_i(\omega)\in\R,\omega\mapsto d_i(\omega)\in\{1,\cdots,d\}$ and $\omega\mapsto E_i(\omega)$ are measurable.
\end{lemma}
\begin{remark}\label{hyperbolic}
We use $\Omega$ to denote $\tilde{\Omega}$ and assume that $\theta$ is ergodic, then we can write $p(\omega),\lambda_i(\omega),d_i(\omega)$ as $p,\lambda_i,d_i$ respectively.
We call $E_i(\omega)$'s Oseledets spaces corresponding to Lyapunov exponents $\lambda_i$ with multiplicities $d_i$. The decomposition $\R^n=E_1(\omega)\oplus E_2\oplus\cdots\oplus E_p(\omega)$ is called an Oseledets splitting. Moreover, $\{\lambda_1,\cdots,\lambda_p;d_1,\cdots,d_p\}$ is called the Lyapunov spectrum. When all Lyapunov exponents are non-zero, we call the linear stochastic system $\Phi(t,\omega)$ hyperbolic.
\end{remark}

Notice that $\R^n=\R^{d_1}\times\cdots\times\R^{d_p}$. For all $i=1,\cdots,p$, let $\pi_i:\R^n\rightarrow\R^{d_i}$ be the projections and $x_i:= \pi_i\mathbf{x}_i$ for all $\mathbf{x}_i\in E_i$, where $E_i =\{0\}\times\cdots\times\{0\}\times\R^{d_i}\times\{0\}\times\cdots\times\{0\}$ and $\mathbf{x}_i =(0,\cdots,0, x_i, 0, \cdots, 0)$. Then the expression of direct sum can be rewritten as
$$\mathbf{x} = (x_1,\cdots,x_p);~x_1\in\R^{d_1},\cdots, x_p\in\R^{d_p}.$$

\begin{lemma}\label{LL+05,1}
\textbf{(Lemma 2.8, \cite{[LL+05]})}
Let $\Phi$ be the linear RDS given in the Multiplicative Ergodic Theorem (Lemma \ref{Oseledets theorem}). There is an invertible measurable mapping $B:\Omega\rightarrow Gl(n,\mathbb{R})$ such that\\
(i) $\Phi(t,\omega)$ is conjugate to a block-diagonal RDS $\Psi(t,\omega)$ by $B(\omega)$, i.e.,
$$B(\theta^t\omega)\Phi(t,\omega)B^{-1}(\omega)=\Psi(t,\omega)=\text{diag}(\Psi_1(t,\omega),\cdots, \Psi_p(t,\omega)),$$
where $\Psi_i(t,\omega))$ are cocycles on $\R_{d_i}$;\\
(ii) The random transformation $B(\omega)$ preserves the Lyapunov spectrum $\{(\lambda_i
, d_i) | 1\leq i\leq p\}$ and the corresponding Oseledets spaces;\\
(iii) Both $||B(\omega)||$ and $||B^{-1}(\omega)||$ are tempered (Definition \ref{tempered}).
\end{lemma}
For simplicity, we still use $\Phi,\Phi_i$ for $\Psi,\Psi_i$. The linear RDS $\Phi(t,\omega)$ is a block-diagonal form according to Lemma \ref{LL+05,1} with
$$\Phi=\left[\begin{matrix}
\Phi_1&0&\cdots&0\\
0&\ddots& &\vdots\\
\vdots& &\ddots&0\\
0&\cdots&0&\Phi_p
\end{matrix}\right],
$$
where $\Phi_i(t,\omega):=\pi_i\circ\Phi(t,\omega)\circ\pi_i^{-1}$ maps $\R^{d_i}$ into itself, $\pi_i^{-1}x_i:=\mathbf{x}_i$ and $\Phi_i(t,\omega)$ satisfies the following lemma.
\begin{lemma}\label{LL+05,2}
\textbf{(Proposition 4.3.3, \cite{[A+98]})}
For each $\epsilon>0$, there is a tempered random variable $M_\epsilon(\omega):\Omega\rightarrow[1,\infty)$ such that\\
(i) $M_\epsilon(\theta^t\omega)\leq M_\epsilon(\omega)e^{\epsilon|t|}$;\\
(ii) $||\Phi_i(t,\omega)||\leq M_\epsilon(\omega)e^{\lambda_it+\epsilon|t|}$.
\end{lemma}

\subsection{A steady linear system and its output system}
\begin{example} Consider
\begin{equation}\label{2linear}
	\begin{aligned}	
	&\left\{  \begin{aligned}
		dX_t &=AX_tdt+BdW_t,\\
		          X_0&=x_0 ,\\ \end{aligned}  \right.\\
	&\left\{  \begin{aligned}
		dY_t &=CX_tdt+DdW_t,\\
	             Y_0&=y_0,\\  \end{aligned}  \right.
    \end{aligned}
\end{equation}
where $A$ and $C$ are nonsingular $n$-order constant matrices, while $B$ and $D$ are constant matrices of $n\times d$ with $CA^{-1}B=D$.
\end{example}

\begin{proposition}
Suppose that $X(t,x_0 )$ and $Y(t,y_0 )$ satisfy system (\ref{2linear}). Then
there is a homeomorphic mapping $K^*=CA^{-1}$ such that $X(t,x_0 )$ and $Y(t,y_0 )$ are completely similar (conjugate). Further, the similarity degree of SDES in (\ref{2linear}) is 1.
\end{proposition}
\begin{proof}
Let $y_0=Kx_0$, then $K(X(t,x_0))-Y(t,K(x_0))$ satisfies the following SDE:
\begin{equation}\label{2linear-Psi}
d(KX_t-Y_t)=(KA-C)X_tdt+(KB-D)dW_t,~0\leq t\leq T.
\end{equation}
Using It\^{o}'s formula (Lemma \ref{Ito formula}), we can obtain
$$\begin{aligned}
d||KX_t-Y_t||_{\mathcal{X}}^2=&2\langle (KA-C)X_t,KX_t-Y_t\rangle_{\mathcal{X}} dt\\
&+2\langle KX_t-Y_t,(KB-D)dW_t \rangle_{\mathcal{X}}\\
&+||KB-D||^2_{\mathcal{L}^2}dt.
\end{aligned}
$$
Notice that $KX_0-Y_0=\mathbf{0}$, $||KX_0-Y_0||_{\mathcal{X}}^2=0$. Then
$$\begin{aligned}
\mathbb{E}||KX_t-Y_t||_{\mathcal{X}}^2=&\mathbb{E}\int_0^t2\langle (KA-C)X_s,KX_s-Y_s\rangle_{\mathcal{X}} ds\\
&+\mathbb{E}\int_0^t||KB-D||^2_{\mathcal{L}^2}ds.
\end{aligned}
$$

Obviously, $K^*=CA^{-1}$ is a homeomorphic mapping. Then
$$\begin{aligned}
\mathbb{E}||K^*X_t-Y_t||_{\mathcal{X}}^2&=\int_0^t||K^*B-D||^2_{\mathcal{L}^2}ds\\
&=\int_0^t||CA^{-1}B-D||^2_{\mathcal{L}^2}ds\\
&=0,~t\in[0,T].
\end{aligned}$$
According to Definition \ref{conjugation}, $X(t,x_0 )$ and $Y(t,y_0 )$ are conjugate (i.e., completely similar). Further, we know that the similarity degree of SDES (\ref{2linear}) is 1 (Remark  \ref{relationship}).
\end{proof}

\subsection{Two steady linear systems}
\begin{example} Consider
\begin{equation}\label{2linears}
	\begin{aligned}	
	&\left\{  \begin{aligned}
		dX_t &=AX_tdt+BdW_t,\\
		          X_0&=x_0 ,\\ \end{aligned}  \right.\\
	&\left\{  \begin{aligned}
		dY_t &=CY_tdt+DdW_t,\\
	             Y_0&=y_0,\\  \end{aligned}  \right.
    \end{aligned}
\end{equation}
where $A$ and $C$ are $n$-order constant matrices, while $B$ and $D$ are constant matrices of $n\times d$.
\end{example}

Let $\tilde{X}(t,\omega),\tilde{Y}(t,\omega)$ be two linear random dynamical systems of (\ref{2linears}) in $\R^n$ on the probability space $(\Omega,\mathcal{F},\mathbb{P})$, and $\tilde{X_t},\tilde{Y_t}$ are the solutions of the determined part (only including drift term):
\begin{equation}\label{tilde_X,tilde_Y}
	\begin{aligned}	
	&\left\{  \begin{aligned}
		d\tilde{X_t} &=A\tilde{X_t}dt,\\
		          \tilde{X_0}&=x_0 ,\\ \end{aligned}  \right.\\
	&\left\{  \begin{aligned}
		d\tilde{Y_t} &=C\tilde{Y_t}dt,\\
	             \tilde{Y_0}&=y_0.\\  \end{aligned}  \right.
    \end{aligned}
\end{equation}
Let $y_0=Kx_0$, then $K(X(t,x_0))-Y(t,K(x_0))$ satisfies the following SDE:
\begin{equation}\label{2linears-Psi}
d(KX_t-Y_t)=(KAX_t-CY_t)dt+(KB-D)dW_t,~0\leq t\leq T.
\end{equation}
Using It\^{o}'s formula (Lemma \ref{Ito formula}), we can obtain
\begin{equation}\label{2linears-Psi2}
\begin{aligned}
\mathbb{E}||KX_t-Y_t||_{\mathcal{X}}^2=&\mathbb{E}\int_0^t2\langle KAX_s-CY_s,KX_s-Y_s\rangle_{\mathcal{X}} ds\\
&+\mathbb{E}\int_0^t||KB-D||^2_{\mathcal{L}^2}ds.
\end{aligned}
\end{equation}

\begin{proposition}
Suppose that $X(t,x_0 )$ and $Y(t,y_0 )$ satisfy system (\ref{2linears}), there is a homeomorphic mapping $K^*$ which satisfies $K^*B=D$ and the diffusion terms are non-degenerate. Then we have:\\
(i) If $A=C$ is nonsingular, %or the Dissipative assumption (HD) holds, %then $X(t,x_0 )$ and $Y(t,K^*(x_0))$ satisfy asymptotic similarity;
then $X(t,x_0 )$ and $Y(t,K^*(x_0))$ are completely similar (conjugate); \\
%(ii) If the Dissipative assumption (HD) holds, then $X(t,x_0 )$ and $Y(t,K^*(x_0))$ satisfy asymptotic similarity;\\
(ii) If $\tilde{X}(t,\omega)$ and $\tilde{Y}(t,\omega)$ have the same Lyapunov exponent, then $X(t,x_0 )$ and $Y(t,K^*(x_0))$ satisfy asymptotic similarity. If $\tilde{X}(t,\omega)$ and $\tilde{Y}(t,\omega)$ do not have the same Lyapunov exponent and at least one of the different Lyapunov exponents is positive, then $X(t,x_0 )$ and $Y(t,K^*(x_0))$ cannot satisfy asymptotic similarity;\\
(iii) If the Ergodicity assumption (HE) holds, then the cost functional (\ref{conjugation functional}) of SDEs (\ref{2linears}) has the minimum value $J[K^*(\cdot)]=\underset{K(\cdot)\in U[0,T]}{\inf}J[K(\cdot)]$, and the similarity degree $\rho(J[K^*(\cdot)])$ satisfies $\underset{T\rightarrow+\infty}{\lim}\rho(J[K^*])=1$.
\end{proposition}

\begin{proof}

(i) %According to Theorem \ref{existence-Dissipation}, if the Dissipative assumption (HD) holds, $X(t,x_0 )$ and $Y(t,K^*(x_0))$ are conjugate.
Let $A = C$ be nonsingular. According to Lemma \ref{LL+05,2} and $K^*B = D$,
$$||\widetilde{(KX-Y)}(t,\omega)||=||e^{At}||\leq M_\epsilon(\omega)e^{(\lambda_1+\epsilon)t},~\forall~t\geq0.$$
Then applying It\^{o}'s formula to $e^{ct}||K^*X_t-Y_t||^2_{\mathcal{X}}$, we have
$$\begin{aligned}
\mathbb{E}e^{ct}||K^*X_t-Y_t||_{\mathcal{X}}^2=&\mathbb{E}\int_0^tce^{cs}||K^*X_s-Y_s||_{\mathcal{X}}^2ds\\
&+\mathbb{E}\int_0^te^{cs}2\langle A(K^*X_s-Y_s),K^*X_s-Y_s\rangle_{\mathcal{X}} ds\\
\leq&\mathbb{E}\int_0^tce^{cs}||K^*X_s-Y_s||_{\mathcal{X}}^2ds\\
&+\mathbb{E}\int_0^te^{cs}2|| A(K^*X_s-Y_s)||_{\mathcal{X}}\cdot||K^*X_s-Y_s||_{\mathcal{X}} ds\\
\leq&\mathbb{E}\int_0^tce^{cs}||K^*X_s-Y_s||_{\mathcal{X}}^2ds\\
&+\mathbb{E}\int_0^te^{cs} 2(\frac{1}{s}\log(M_\epsilon(\omega))+\lambda_1+\epsilon)||K^*X_s-Y_s||_{\mathcal{X}}^2ds\\
%\mathbb{E}\int_0^te^{cs} 2\lambda_{max}||K^*X_s-Y_s)||^2_{\mathcal{X}} ds\\
%&=(c+2\lambda_{max})\int_0^t\mathbb{E}e^{cs}||KX_s-Y_s)||^2_{\mathcal{X}} ds,~t\in[0,T],
=&\int_0^t(c+2(\frac{1}{s}\log(M_\epsilon(\omega))+\lambda_1+\epsilon))\mathbb{E}e^{cs}||K^*X_s-Y_s||^2_{\mathcal{X}} ds,~t\in[0,T],
%\\&\triangleq\int_0^t\beta(s)\mathbb{E}e^{cs}||KX_s-Y_s||^2_{\mathcal{X}}ds+\alpha(t),~t\in[0,T],
\end{aligned}$$
where $c>-2(\lambda_1+\epsilon),~M_\epsilon(\omega):\Omega\rightarrow[1,\infty)$ is tempered which is defined in Lemma \ref{LL+05,2} and $\lambda_1$ is the largest Lyapunov exponent which is defined in Lemma \ref{Oseledets theorem}.
%. , and $c<-\lambda_1-\epsilon$.
%$c>-2\lambda_{max}$ and $\lambda_{max}$ is the maximum eigenvalue of $A$.
Apply the Gronwall lemma (Lemma \ref{Gronwall lemma}), where $\beta(s)=c+2(\frac{1}{s}\log(M_\epsilon(\omega))+\lambda_1+\epsilon)$ is non-negative and $\alpha(t)=0$. Then $$0\leq\mathbb{E}e^{ct}||K^*X_t-Y_t||_{\mathcal{X}}^2\leq\alpha(t)\exp\bigg(\int_0^t\beta(s)ds\bigg)=0.$$
%, i.e., $0\leq\mathbb{E}||KX_t-Y_t||_{\mathcal{X}}^2\leq0$.
Hence, the equation above has only zero solution:
$$\mathbb{E}||K^*X_t-Y_t||_{\mathcal{X}}^2=0,~t\in[0,T],$$
i.e., $X(t,x_0 )$ and $Y(t,K^*(x_0))$ are conjugate (i.e., completely similar).\\

(ii) If $\tilde{X}(t,\omega)$ and $\tilde{Y}(t,\omega)$ have the same Lyapunov exponent, according to Oseledets theorem (Lemma \ref{Oseledets theorem}), for $x\in E_i(\omega)\backslash\{0\},i=1,2,\cdots,p,$
$$\underset{T\rightarrow\infty}{\lim}\frac{1}{T}\log\frac{||\tilde{X}(T,\omega)x||}{||x||}=\lambda_i
=\underset{T\rightarrow\infty}{\lim}\frac{1}{T}\log\frac{||\tilde{Y}(T,\omega)x||}{||x||}.$$
Then
$$\begin{aligned}
\underset{T\rightarrow\infty}{\lim}\mathbb{E}||K^*(X(T,x_0 ))-Y(T,K^*(x_0))||_{\mathcal{X}}&=
\underset{T\rightarrow\infty}{\lim}||K^*(\tilde{X}(T,x_0 ))-\tilde{Y}(T,K^*(x_0))||_{\R^n}\\
&\leq\underset{T\rightarrow\infty}{\lim}\underset{i=1}{\overset{p}{\sum}}||K^*(\tilde{X}(T,x_0 ))-\tilde{Y}(T,K^*(x_0))||_{E_i(\omega)}\\
&=\underset{T\rightarrow\infty}{\lim}\underset{i=1}{\overset{p}{\sum}}||K^*(x_0e^{\lambda_iT})-K^*(x_0)e^{\lambda_iT}||_{E_i(\omega)}\\
&=0,
\end{aligned}$$
i.e., $X(t,x_0 )$ and $Y(t,K^*(x_0))$ satisfy asymptotic similarity.\\

On the other hand, without loss of generality, according to Oseledets theorem (Theorem \ref{Oseledets theorem}), let
$$\begin{aligned}
\underset{T\rightarrow\infty}{\lim}\frac{1}{T}\log\frac{||\tilde{X}(T,\omega)x||}{||x||}&=\lambda_1,\\
\underset{T\rightarrow\infty}{\lim}\frac{1}{T}\log\frac{||\tilde{Y}(T,\omega)x||}{||x||}&=\lambda_2,
\end{aligned}$$
where $x\in E_1(\omega)\backslash\{0\}$ and $\lambda_1\neq\lambda_2,~\lambda_1>0$. Then
$$\begin{aligned}
\underset{T\rightarrow\infty}{\lim}\mathbb{E}||K^*(X(T,x_0 ))-Y(T,K^*(x_0))||_{\mathcal{X}}&=
\underset{T\rightarrow\infty}{\lim}||K^*(\tilde{X}(T,x_0 ))-\tilde{Y}(T,K^*(x_0))||_{\R^n}\\
&\geq\underset{T\rightarrow\infty}{\lim}||K^*(\tilde{X}(T,x_0 ))-\tilde{Y}(T,K^*(x_0))||_{E_1(\omega)}\\
&=\underset{T\rightarrow\infty}{\lim}||K^*(x_0e^{\lambda_1T})-K^*(x_0)e^{\lambda_2T}||_{E_1(\omega)}\\
&>0,
\end{aligned}$$
i.e., $X(t,x_0 )$ and $Y(t,K^*(x_0))$ cannot satisfy asymptotic similarity.\\

(iii) Obviously, the Ergodicity assumption (HE) holds when the diffusion term is nondegenerate, i.e., $KX_t-Y_t$ defined in (\ref{2linears-Psi}) is ergodic.
Then according to Theorem \ref{existence-Ergodicity},
$$\begin{aligned}
\underset{T\rightarrow+\infty}{\lim}J[K^*]&=\underset{T\rightarrow\infty}{\lim}\bigg[\dfrac{1}{T} \int_{0}^{T}\mathbb{E}||K^*(X(t,x_0))-Y(t,K^*(x_0))||_{\mathcal{X}}^2 dt\bigg]\\
&=\underset{T\rightarrow+\infty}{\lim}\frac{1}{T}\int_0^T\phi(X_t,Y_t))dt\\
&=\underset{m\rightarrow\infty}{\lim}\frac{1}{m}\underset{i=0}{\overset{m-1}{\sum}}\phi(\Psi^i(x,y))\\
&=\int\phi(K^*x_0-K^*x_0) d\mu\\
&=0,~\mu-a.e..
\end{aligned}$$
where $\Psi^i$ is defined in (\ref{tau_i}).
Hence, the cost functional (\ref{conjugation functional}) has the minimum value $J[K^*(\cdot)]=\underset{K(\cdot)\in U[0,T]}{\inf}J[K(\cdot)]$, and $K^*$ is the minimizer. We can define $\rho(J[K^*])$ (Definition \ref{similarity}) to describe the similarity between two SDEs (\ref{2linears}), and $\underset{T\rightarrow+\infty}{\lim}\rho(J[K^*])=1$.\\
\end{proof}

\subsection{A nonlinear system and its linearization system}
\begin{example} Consider a nonlinear SDE and a linearized SDE:
\begin{equation}\label{linear and nonlinear}
	\begin{aligned}	
	&\left\{  \begin{aligned}
		dX_t &=f_0(X_t)dt+\underset{l=1}{\overset{d}{\sum}}f_l(X_t)dW_t^l,\\
&\triangleq f_0(X_t)dt+\sigma(X_t) dW_t,\\
		          X_0&=x_0 ,\\ \end{aligned}  \right.\\
	&\left\{  \begin{aligned}
		dY_t &=A_0Y_tdt+\underset{l=1}{\overset{d}{\sum}}A_lY_tdW_t^l,\\
	             Y_0&=y_0,\\  \end{aligned}  \right.
    \end{aligned}
\end{equation}
where $\sigma(X_t)=(f_1,\cdots,f_d)$ is a matrix of $n\times d$, $W_t^l(l=1,\cdots,d)$ is a 1-dimensional standard Brownian motion, $A_0=\frac{\partial}{\partial X}f_0(\mathbf{0})$, $A_l=\frac{\partial}{\partial X}f_l(\mathbf{0})$, and $\mathbf{0}$ is a fixed point for the vector fields $f_0,f_1,\cdots,f_d$.
Define the function space $C_b^2$ as the set of functions $f:\R^n\rightarrow\R^n$, where all derivatives up to order $2$ exist and are bounded.
Suppose that $f_0,f_1,\cdots,f_d\in C_b^2$
and $\sigma(X_t) dW_t$ is non-degenerate noise (i.e., $\sigma(X_t)\sigma^T(X_t)=\underset{l=1}{\overset{d}{\sum}}f_l(X_t)f_l^T(X_t)$ is invertible).
\end{example}

Let $\Psi$ and $\Phi$ be the cocycle generated by the nonlinear and linearized SDE of (\ref{linear and nonlinear}), respectively. %Suppose that $\Psi$ is hyperbolic, i.e., Lyapunov exponents are nonzero (Remark \ref{hyperbolic}).
Write the time-one mapping $\Psi(1,\omega,x)$ as $\varphi(\omega,x)$ and $\Phi(1,\omega,y)$ as $A(\omega)y$.
Suppose that $\varphi(\omega,x)$ has a fixed point $x=\mathbf{0}$ for all $\omega\in\Omega$, and is locally tempered $C^1$ random diffeomorphism, that is, there is a tempered ball $B_{\delta(\omega)}=\{x|~||x||<\delta(\omega)\}$, where $\delta(\omega)$ is tempered from below (Definition \ref{tempered} (iii)), such that
$$\underset{x\in B_{\delta(\omega)}}{\sup}||D^i_x\varphi(\omega,x)||=C_i(\omega),$$
where $C_i(\omega),i=0,1$ is tempered from above (Definition \ref{tempered} (ii)).

The cocycle $\Phi(t,\omega)$ generated by the linearized SDE is
$$\Phi(t,\omega)=e^{(A_0-\frac{1}{2}\underset{l=1}{\overset{d}{\sum}}A_l^2)t
+\underset{l=1}{\overset{d}{\sum}}A_lW_t^l(\omega)}.$$
Suppose that $A(\omega):=D\varphi(\omega,0)$ satisfies the conditions of the Oseledets theorem (Lemma \ref{Oseledets theorem}) with the Lyapunov exponents $\lambda_p<\cdots<\lambda_1<0$ (or $0<\lambda_p<\cdots<\lambda_1$), then $||\Phi(t,\omega)||\leq M_\epsilon(\omega)e^{(\lambda_1+\epsilon)t},~\forall~t\geq0$ (Lemma \ref{LL+05,2}), where $M_\epsilon(\omega):\Omega\rightarrow[1,\infty)$ is tempered.
%and $\lambda_1$ is the largest Lyapunov exponent.
We only need to discuss the case that $\lambda_p<\cdots<\lambda_1<0$, the other case can be dealt with by considering the inverse of the system (i.e., let $t=-t$).

The similarity (conjugacy) between two SDEs in (\ref{linear and nonlinear}) is reflected by the well-known stochastic Hartman-Grobman theorem, due to (\cite{[IL+02],[LZZ+20],[HP+16],[ZLZ+17]}):
\begin{lemma}\label{stochastic Hartman-Grobman theorem}
\textbf{(A stochastic Hartman-Grobman theorem)}
Let $\Psi$ and $\Phi$ be the cocycle generated by the nonlinear and linearized SDE of (\ref{linear and nonlinear}), respectively.
Suppose that $\varphi:\Omega\times B_{\delta(\omega)}\rightarrow\R^n$ is a tempered $C^1$ random diffeomorphism such that $\varphi(\omega,\mathbf{0})=\mathbf{0}$.
Suppose that $A(\omega)$ satisfies the conditions of the Oseledets theorem (Lemma \ref{Oseledets theorem}) with the Lyapunov exponents $\lambda_p<\cdots<\lambda_1<0$.
Then there exists a mapping $H:\Omega\times\R^n\rightarrow\R^n$ such that\\
(i) $H(\omega,\cdot):\R^n\rightarrow\R^n$ is a homeomorphism of $\R^n$, and $H(\omega,0)=0,~\omega\in\Omega$;\\
(ii) the following topological equivalence relation
$$H(\theta_t\omega,\cdot)\circ\Psi_t(\omega)=\Phi_t(\omega)\circ H(\omega,x)$$
holds in a random time interval $t\in[0,\tau_+(\omega,x)]$ for $\omega\in\Omega$.
\end{lemma}

The stochastic Hartman-Grobman theorem (Lemma \ref{stochastic Hartman-Grobman theorem}) gives sufficient conditions for the existence of homeomorphism $H(\omega,\cdot)$,
which is attributed to \cite{[IL+02],[LZZ+20]}.
%which is a conjugate (completely similar) case in our results.
%Here, we will give the form of homeomorphism $H(\omega,\cdot)$
%which satisfies $$H(\theta_t\omega,\cdot)\circ\Psi_t(\omega)\circ H^{-1}(\omega,x)=\Phi_t(\omega)$$
%by using the method proven in the previous section, which is
The stochastic maximum principle (Theorem \ref{maxsde}) provides the necessary conditions for the existence of homeomorphism $H(\omega,\cdot)$, and Proposition \ref{HJB-K} gives equation (\ref{maxsde K^*d}) to solve it.

%Here, we will provide another proof by using the method proven in the previous section.
%To prove this theorem, we need to find a homeomorphism $H(\omega,\cdot)$ such that
%$$H(\theta_t\omega,\cdot)\circ\Psi_t(\omega)\circ H^{-1}(\omega,x)=\Phi_t(\omega).$$

Actually, let $K^*=H(\omega,\cdot)$, according to Definition \ref{conjugation},
$$\begin{aligned}
\mathbb{E}||K^*X_t-Y_t||^2&=\mathbb{E}||H(\theta_t\omega,\cdot)\Psi_t(\omega)x_0-\Phi_t(\omega)y_0||^2\\
&=\mathbb{E}||H(\theta_t\omega,\cdot)\Psi_t(\omega)x_0-\Phi_t(\omega)H(\omega,\cdot)x_0||^2\\
&=0.
\end{aligned}$$
Thus, the homeomorphism $H(\omega,\cdot)$ in the stochastic Hartman-Grobman theorem (Lemma \ref{stochastic Hartman-Grobman theorem}) is $K^*$ in this paper, and $X(t,x_0 )$ and $Y(t,K^*(x_0))$ are conjugate (i.e., completely similar).

%Then, we solve $K^*$ which satisfies the stochastic maximum principle (Theorem \ref{maxsde}). Notice the condition that $\mathbf{0}$ is a fixed point for the vector fields $f_0,f_1,\cdots,f_d$.
%There is a random time interval $[0,\tau_+(\omega)]$ such that for all $\epsilon>0$, $||X_t||<\delta(\omega)$, where $\delta(\omega)$ is tempered from below (Definition \ref{tempered} (iii)).

\begin{remark}
Since the random time interval is space-dependent, the topological equivalence in Lemma \ref{stochastic Hartman-Grobman theorem} holds locally in space. In other words, near the fixed point, two SDEs of (\ref{linear and nonlinear}) are conjugate (completely similar).
\end{remark}

Finally, we provide a proof of the sufficient existence of $K^*$, which is mainly inspired by \cite{[HP+16]}, and use the HJB equation in Proposition \ref{HJB-K} to provide the necessary conditions that $K^*$ satisfies.

As Definition 2.2 in \cite{[HP+16]}, for a fixed $\omega\in\Omega$, let $P_+\in\mathcal{B}(\mathcal{X})$ be the dichotomic projection with $P_-=I_{\mathcal{X}}-P_+$. Define the Green kernel corresponding to the dichotomy as:
$$
G_A(t)=\left\{\begin{aligned}
\Phi(t,\omega)P_+,t\geq0,\\
-\Phi(t,\omega)P_-,t<0.
\end{aligned}\right.
$$
It is obvious that $G_A:\R\rightarrow\mathcal{B}(\mathcal{X})$ is strongly continuous in $t\neq0$, and $G_A$ is strongly $L_1$, with estimate $||G_A(t)||_{\mathcal{B}(\mathcal{X})}\leq M_\epsilon(\omega)e^{(\lambda_1+\epsilon)|t|}$ (Lemma \ref{LL+05,2}), where $M_\epsilon(\omega):\Omega\rightarrow[1,\infty)$ is tempered.

For all $t\in[0,\tau_+(\omega)]$, we express $f_l,l=0,\cdots,d$ by Taylor expansion:
$$f_l(X_t)=f_l(\mathbf{0})+A_l(X_t-\mathbf{0})+o(X_t-\mathbf{0})=A_lX_t+o(X_t).$$
Then, the first equation of (\ref{linear and nonlinear}) is changed into
$$
\left\{  \begin{aligned}
		\frac{dX_t}{dt}&=A_0X_t+\underset{l=1}{\overset{d}{\sum}}A_lX_t\frac{dW_t^l}{dt}+
o(X_t)+\underset{l=1}{\overset{d}{\sum}}o(X_t)\frac{dW_t^l}{dt},\\
&\triangleq A_0X_t+\underset{l=1}{\overset{d}{\sum}}A_lX_t\frac{dW_t^l}{dt}+
\gamma(X(t,x_0)),\\
		          X_0&=x_0,\\ \end{aligned}  \right.
$$
where
$$\begin{aligned}
\gamma(X(t,x_0))&=o(X_t)+\underset{l=1}{\overset{d}{\sum}}o(X_t)\frac{dW_t^l}{dt}\\
&=\frac{1}{2}\langle\frac{\partial^2f_0(\xi)}{\partial X^2}X_t,X_t\rangle
+\underset{l=1}{\overset{d}{\sum}}\frac{1}{2}\langle\frac{\partial^2f_l(\xi)}{\partial X^2}X_t\frac{dW_t^l}{dt},X_t\rangle,\xi\in B_{\delta(\omega)}.
\end{aligned}$$
Then
$$\mathbb{E}\underset{t\in[0,\tau_+(\omega)]}{\sup}||\gamma(\cdot)||\leq\underset{t\in[0,\tau_+(\omega)]}{\sup}||X_t||\leq\delta(\omega),$$
and for all $x_1,x_2\in B_{\delta(\omega)}$,
$$\begin{aligned}
&\mathbb{E}||\gamma(x_1)-\gamma(x_2)||\\
=&\mathbb{E}||\frac{1}{2}\langle\frac{\partial^2f_0(\xi_1)}{\partial X^2}x_1,x_1\rangle
-\frac{1}{2}\langle\frac{\partial^2f_0(\xi_2)}{\partial X^2}x_2,x_2\rangle\\
&+\underset{l=1}{\overset{d}{\sum}}\frac{1}{2}\langle\frac{\partial^2f_l(\xi_1)}{\partial X^2}x_1\frac{dW_t^l}{dt},x_1\rangle
-\underset{l=1}{\overset{d}{\sum}}\frac{1}{2}\langle\frac{\partial^2f_l(\xi_2)}{\partial X^2}x_2\frac{dW_t^l}{dt},x_2\rangle||\\
\leq&\mathbb{E}\frac{1}{2}\mathfrak{C}\big| ||x_1||^2-||x_2||^2\big|\\
\leq&\mathbb{E}\frac{1}{2}\mathfrak{C}(||x_1||+||x_2||)\big| ||x_1||-||x_2||\big|\\
\leq&\mathfrak{C}\delta(\omega)||x_1-x_2||\\
\triangleq&c_1(\omega)||x_1-x_2||,
\end{aligned}$$
where $\mathfrak{C}=\underset{l=0,\cdots,d}{\max}~\underset{t\in[0,\tau_+(\omega)]}{\sup}||\frac{\partial^2f_l(X_t)}{\partial X^2}||$ and $\delta(\omega)=\frac{-(\lambda_1+\epsilon)}{4\mathfrak{C}M_\epsilon(\omega)}$. Hence,
%$\mathbb{E}\underset{t\in[0,\tau_+(\omega)]}{\sup}||\gamma(\cdot)||\leq\delta(\omega)$, $\mathbb{E}||\gamma(x_1)-\gamma(x_2)||\leq c_1(\omega)||x_1-x_2||,~\forall x_1,x_2\in B_{\delta(\omega)}$, and
\begin{equation}\label{construction condition}
\frac{2M_\epsilon(\omega)c_1(\omega)}{-(\lambda_1+\epsilon)}=\frac{1}{2}<1,~\forall\epsilon>0.
\end{equation}

\begin{theorem}\label{stochastic Hartman-Grobman theorem1}
Suppose that the conditions of Lemma \ref{stochastic Hartman-Grobman theorem} and (\ref{construction condition}) hold. Then there is a homeomorphism $K^*$ such that
$$\mathbb{E}||K^*(X(t,x_0))-Y(t,K^*(x_0)||=0,~t\in[0,\tau_+(\omega)],$$
and
$$\frac{\partial K^*}{\partial X^T}=\bigg[\underset{l=1}{\overset{d}{\sum}}A_lYf_l^T(X)-f_0(X)(K^*X-Y)^T\bigg]
\bigg[\underset{l=1}{\overset{d}{\sum}}f_l(X)f_l^T(X)\bigg]^{-1}.$$
\end{theorem}
\begin{proof}
The proof is divided into three steps. \\

\textbf{Step 1. Construct the map $K^*$.}

Suppose that $K^*=I_{\mathcal{X}}-\hat{\kappa}$ satisfies $\mathbb{E}K^*(X(t,x_0))=\mathbb{E}Y(t,K^*(x_0))$, and is invertible with $(K^*)^{-1}=I_{\mathcal{X}}+\kappa$. Notice that $(K^*)^{-1}(y)=y+\kappa(y),~\forall y\in B_{\delta(\omega)}$.

Taking the derivative on both sides of the equation $\mathbb{E}X(t,x_0)=\mathbb{E}(K^*)^{-1}(Y(t,K^*(x_0)))$, we have
$$\mathbb{E}\frac{dX(t,(K^*)^{-1}y_0)}{dt}=\mathbb{E}\frac{(K^*)^{-1}(Y(t,K^*(x_0)))}{dY}\frac{dY(t,K^*(x_0))}{dt},$$
i.e.,
$$\begin{aligned}
&\mathbb{E}\big[A_0(K^*)^{-1}(Y(t,K^*(x_0)))+\underset{l=1}{\overset{d}{\sum}}A_l(K^*)^{-1}(Y(t,K^*(x_0)))\frac{dW_t^l}{dt}+\gamma((K^*)^{-1}(Y(t,K^*(x_0))))\big]\\
=&\mathbb{E}\big[A_0(Y_t+\kappa(Y_t))+\underset{l=1}{\overset{d}{\sum}}A_l(Y_t+\kappa(Y_t))\frac{dW_t^l}{dt}+\gamma((Y_t+\kappa(Y_t)))\big]\\
=&\mathbb{E}\big[(I+\kappa'(Y))(A_0Y_t+\underset{l=1}{\overset{d}{\sum}}A_lY_t\frac{dW_t^l}{dt})\big].
\end{aligned}$$

Let $w(t)=\kappa(Y_t)$, then
$$\mathbb{E}\big[\frac{dw(t)}{dt}\big]=\mathbb{E}\big[A_0w(t)+\underset{l=1}{\overset{d}{\sum}}A_lw(t)\frac{dW_t^l}{dt}\gamma((Y_t+\kappa(Y_t)))\big].$$
Due to the definition of $w$, it is obvious that $w$ is bounded if $\kappa$ is bounded. Furthermore, the unique mild bounded solution of the the above equation is given by
$$\mathbb{E}w(t)=\mathbb{E}\big[\int_\R G_A(s)\gamma(\Phi(t-s,\omega)y_0+\kappa(\Phi(t-s,\omega)y_0))ds\big].$$
Setting $t=0$ yields the following functional equation for the map $\kappa$:
$$\mathbb{E}\big[\kappa(y_0)\big]=\mathbb{E}\big[\int_\R G_A(s)\gamma(\Phi(-s,\omega)y_0+\kappa(\Phi(-s,\omega)y_0))ds\big].$$

Define
$$\mathbb{E}\kappa(\cdot)\mapsto\mathbb{E}\int_\R G_A(s)\gamma(\Phi(-s,\omega)\cdot+\kappa(\Phi(-s,\omega)\cdot))ds:=
(\mathcal{T}\mathbb{E}\kappa)(\cdot),~\kappa\in BUC(\mathcal{X}),$$
where $BUC(\mathcal{X})$ is the bounded uniformly continuous function space on $\mathcal{X}$. \\

\textbf{Step 2. Existence and Uniqueness of $\kappa$.}

Suppose that the condition (\ref{construction condition}) holds. Then $\mathcal{T}$ is a selfmap on $BUC(\mathcal{X})$, and there is a unique fixed point $\kappa\in BUC(\mathcal{X})$ of $\mathcal{T}$.

Firstly, for a bounded function $\kappa\in BUC(\mathcal{X})$ , $\mathcal{T}\mathbb{E}\kappa$ is bounded as well, because of
$$\begin{aligned}
||\mathcal{T}\mathbb{E}\kappa||=&||\mathbb{E}\int_\R G_A(s)\gamma(\Phi(-s,\omega)\cdot+\kappa(\Phi(-s,\omega)\cdot))ds||\\
\leq&\mathbb{E}\int_\R ||G_A(s)||~||\gamma(\Phi(-s,\omega)\cdot+\kappa(\Phi(-s,\omega)\cdot))||ds\\
\leq&\mathbb{E}\int_\R M_\epsilon(\omega)e^{(\lambda_1+\epsilon)|s|}\delta(\omega)ds\\
=&\frac{2M_\epsilon(\omega)\delta(\omega)}{-(\lambda_1+\epsilon)}.\\
\end{aligned}$$

Secondly, let $x_1,x_2\in B_{\delta(\omega)}$. For all $\epsilon>0$, there exists $N=[\frac{1}{(\lambda_1+\epsilon)}\log\frac{-\epsilon(\lambda_1+\epsilon)}{12M_\epsilon(\omega)\delta(\omega)}]+1\in\N$ and $\delta=\frac{-\epsilon(\lambda_1+\epsilon)}{6M_\epsilon(\omega)c_1(\omega)(1-e^{(\lambda_1+\epsilon)N})}$ such that $||x_1-x_2||\leq\delta$, $|s|\leq N$,
$$||\kappa(\Phi(-s,\omega)x_1)-\kappa(\Phi(-s,\omega)x_2)||<\frac{\epsilon}{6NM_\epsilon(\omega)c_1(\omega)}.$$
From condition (\ref{construction condition}), we obtain
$$\begin{aligned}
&||\mathcal{T}\mathbb{E}\kappa(x_1)-\mathcal{T}\mathbb{E}\kappa(x_2)||\\
\leq&\mathbb{E}\int_\R M_\epsilon(\omega)e^{(\lambda_1+\epsilon)|s|}
||\gamma(\Phi(-s,\omega)x_1+\kappa(\Phi(-s,\omega)x_1))-\gamma(\Phi(-s,\omega)x_2+\kappa(\Phi(-s,\omega)x_2))||ds\\
\leq&\mathbb{E}\big[2M_\epsilon(\omega)\delta(\omega)\int_{|s|>N}e^{(\lambda_1+\epsilon)|s|}ds\\
&+M_\epsilon(\omega)c_1(\omega)\int_{|s|\leq N}\big(|\Phi(-s,\omega)(x_1-x_2)|
+|\kappa(\Phi(-s,\omega)x_1)-\kappa(\Phi(-s,\omega)x_2)|\big)ds\big]\\
\leq&\mathbb{E}\big[\frac{4M_\epsilon(\omega)\delta(\omega)}{-(\lambda_1+\epsilon)}e^{(\lambda_1+\epsilon)N}
+\frac{2M_\epsilon(\omega)c_1(\omega)(1-e^{(\lambda_1+\epsilon)N})}{-(\lambda_1+\epsilon)}||x_1-x_2||\\
&+2NM_\epsilon(\omega)c_1(\omega)\underset{|s|\leq N}{\sup}||\kappa(\Phi(-s,\omega)x_1)-\kappa(\Phi(-s,\omega)x_2)||\big]\\
<&\frac{\epsilon}{3}+\frac{\epsilon}{3}+\frac{\epsilon}{3}=\epsilon.
\end{aligned}$$
This shows $\mathcal{T}\mathbb{E}\kappa\in BUC(\mathcal{X})$.

Thirdly, for arbitrary $\kappa_1,\kappa_2\in BUC(\mathcal{X})$ and $x\in\mathcal{X}$, the following estimate holds:
$$\begin{aligned}
&||\mathcal{T}\mathbb{E}\kappa_1(x)-\mathcal{T}\mathbb{E}\kappa_2(x)||\\
\leq&\mathbb{E}\int_\R M_\epsilon(\omega)e^{(\lambda_1+\epsilon)|s|}
||\gamma(\Phi(-s,\omega)x+\kappa_1(\Phi(-s,\omega)x))-\gamma(\Phi(-s,\omega)x+\kappa_2(\Phi(-s,\omega)x))||ds\\
\leq&\mathbb{E}M_\epsilon(\omega)c_1(\omega)\int_\R e^{(\lambda_1+\epsilon)|s|}
||\kappa_1(\Phi(-s,\omega)x)-\kappa_2(\Phi(-s,\omega)x)||ds\\
\leq&\mathbb{E}M_\epsilon(\omega)c_1(\omega)||\kappa_1(\cdot)-\kappa_2(\cdot)||_{BUC(\mathcal{X})}
\int_\R e^{(\lambda_1+\epsilon)|s|}ds\\
=&\frac{2M_\epsilon(\omega)c_1(\omega)}{-(\lambda_1+\epsilon)}||\kappa_1(\cdot)-\kappa_2(\cdot)||_{BUC(\mathcal{X})},
\end{aligned}$$
where $\frac{2M_\epsilon(\omega)c_1(\omega)}{-(\lambda_1+\epsilon)}<1$ by (\ref{construction condition}).
This means that $\mathcal{T}$ is a strict contraction on $BUC(\mathcal{X})$. Then the contraction mapping principle yields a unique solution $\kappa\in BUC(\mathcal{X})$. \\

\textbf{Step 3. Solve $K^*$ which satisfies the stochastic maximum principle (Theorem \ref{maxsde}).}
%Notice the condition that $\mathbf{0}$ is a fixed point for the vector fields $f_0,f_1,\cdots,f_d$.
%There is a random time interval $[0,\tau_+(\omega)]$ such that for all $\epsilon>0$, $||X_t||<\delta(\omega)$, where $\delta(\omega)$ is tempered from below (Definition \ref{tempered} (iii)).

Let $\psi_t^K=KX_t-Y_t,~t\in[0,\tau_+(\omega)]$, and it satisfies
\begin{equation}\label{linear-nonlinear-Psi}
\left\{\begin{aligned}
d\psi_t^K=&(\frac{\partial K}{\partial X^T}f_0(X_t)-A_0Y_t)dt+\underset{l=1}{\overset{d}{\sum}}(\frac{\partial K}{\partial X^T}f_l(X_t)-A_lY_t)dW_t^l\\
\triangleq&F(\psi_t^K,t)dt+G(\psi_t^K,t)dW_t,\\
\psi_0^K=&\psi_0,
\end{aligned}\right.
\end{equation}
where $F(\psi_t^K,t)\triangleq\frac{\partial K}{\partial X^T}f_0(X_t)-A_0Y_t$ is an $n$-dimensional vector and
$$G(\psi_t^K,t)=\bigg(\frac{\partial K}{\partial X^T}f_1(X_t)-A_1Y_t,\cdots,\frac{\partial K}{\partial X^T}f_d(X_t)-A_dY_t\bigg)$$
is a matrix of $n\times d$.
The cost functional corresponding to (\ref{cost functional}) is
$$\begin{aligned}
J[K]&=\mathbb{E}\bigg[\int_0^{\tau_+(\omega)}\frac{1}{\tau_+(\omega)}||KX_t-Y_t||^2dt+||KX_{\tau_+(\omega)}-Y_{\tau_+(\omega)}||\bigg]\\
&\triangleq \mathbb{E}\bigg[\int_0^{\tau_+(\omega)} L(\psi_t^K)dt+h(\psi_{\tau_+(\omega)}^K)\bigg].
\end{aligned}$$

Define $V(\psi,t)=\underset{K}{\inf}~\mathbb{E}\big[\int_t^{\tau_+(\omega)} L(\psi_s^K)ds+h(\psi_{\tau_+(\omega)}^K)\big],~t\in[0,\tau_+(\omega)]$. By Proposition \ref{HJB-K}, $V(\psi,t)$ satisfies the following Hamilton-Jacobi-Bellman (HJB) equation:
\begin{equation}\label{HJB-d1}
\left\{\begin{aligned}
&V_t+\underset{K}{\inf}\bigg\{V_\psi F(\psi,t)+\frac{1}{2}Tr\bigg[G^T(\psi,t)V_{\psi\psi}G(\psi,t)\bigg]
+L(\psi)\bigg\}=0,\\
&V(\psi,\tau_+(\omega))=h(\psi_{\tau_+(\omega)}^K),
\end{aligned}\right.
\end{equation}
where $V_t:=\frac{\partial V}{\partial t}(\psi,t)=-L(\psi),~V_\psi:=\frac{\partial V}{\partial \psi}(\psi,t)=2\psi^T,~V_{\psi\psi}:=\frac{\partial^2 V}{\partial \psi^2}(\psi,t)=2I$.
%Let $\psi$ be a solution path, we have $V_\psi=2\psi^T,~V_{\psi\psi}=2I$.

Since the minimizer $K^*$ exists, according to the stochastic maximum principle (Theorem \ref{maxsde}), from (\ref{HJB-d1}) we can obtain
$$\begin{aligned}
&V_\psi F(\psi^{K^*},t)+\frac{1}{2}Tr\bigg[G^T(\psi^{K^*},t)V_{\psi\psi}G(\psi^{K^*},t)\bigg]\\
=&V_\psi\big(\frac{\partial K^*}{\partial X^T}f_0(X)-A_0Y\big)+
\frac{1}{2}\underset{l=1}{\overset{d}{\sum}}\big(\frac{\partial K^*}{\partial X^T}f_l(X)-A_lY\big)^TV_{\psi\psi}\big(\frac{\partial K^*}{\partial X^T}f_l(X)-A_lY\big)\\
=&0.
\end{aligned}$$
Substituting $V_\psi=2\psi^T$ and $V_{\psi\psi}=2I$ into the above equation, we get
\begin{equation}\label{K^* equation1}
2\langle \frac{\partial K^*}{\partial X^T}f_0(X)-A_0Y,K^*X-Y\rangle
+\underset{l=1}{\overset{d}{\sum}}||\frac{\partial K^*}{\partial X^T}f_l(X)-A_lY||^2=0.
\end{equation}
Using It\^{o}'s formula (Lemma \ref{Ito formula}) to $||K^*X_t-Y_t||^2,~t\in[0,\tau_+(\omega)]$, we have
$$\begin{aligned}
\mathbb{E}||K^*X_t-Y_t||^2=&\mathbb{E}\int_0^t2\langle \frac{\partial K^*}{\partial X^T}f_0(X_s)-A_0Y_s,K^*X_s-Y_s\rangle \\
&+\underset{l=1}{\overset{d}{\sum}}||\frac{\partial K^*}{\partial X^T}f_l(X_s)-A_lY_s||^2ds.
\end{aligned}$$
Setting $\psi=\psi_s^{K^*}$ and substituting (\ref{K^* equation1}) into the above equation,  we obtain
$$\mathbb{E}||K^*X_t-Y_t||^2=0,~t\in[0,\tau_+],$$
i.e., $X(t,x_0 )$ and $Y(t,K^*(x_0))$ are conjugate (Definition \ref{conjugation}) in $[0,\tau_+]$.

Since $\sigma(X_t)\sigma^T(X_t)$ is invertible, $\underset{l=1}{\overset{d}{\sum}}f_l(X)f_l^T(X)$ is invertible. Then we can solve from (\ref{HJB-d1}) that the minimizer $K^*$ satisfies the following equation:
$$f_0(X)(K^*X-Y)^T+\underset{l=1}{\overset{d}{\sum}}(\frac{\partial K^*}{\partial X^T}f_l(X)-A_lY)f_l^T(X)=\mathbf{0}_{n\times n},$$
i.e.,
\begin{equation}\label{maxsde K^*d1}
\frac{\partial K^*}{\partial X^T}=\bigg[\underset{l=1}{\overset{d}{\sum}}A_lYf_l^T(X)-f_0(X)(K^*X-Y)^T\bigg]
\bigg[\underset{l=1}{\overset{d}{\sum}}f_l(X)f_l^T(X)\bigg]^{-1}.
\end{equation}

\end{proof}

To the best of our knowledge, this article is the first to examine the relationship between two stochastic differential systems from the perspective of similarity, which can be seen as an extension of stochastic conjugate theory. We transform the problem into the existence problem to the minimizer $K^*$ of the cost functional. Under appropriate assumptions, we have proved the sufficient conditions (a strong law of large numbers) and necessary conditions (a stochastic maximum principle) for the existence of minimizer.

There still remain many open problems to study. For instance, it is difficult to explicitly express the minimizer $K^*$ for general SDEs, and searching for numerical methods to obtain the numerical solution of the minimizer $K^*$ is without doubt a valuable and meaningful work and we will consider this work in future research. Besides, in sufficient conditions, we know that the SLLN holds, and whether the corresponding central limit theorem (CLT) and large deviation principle (LDP) can be obtained. In the stochastic maximum principle, how to explicitly solve the adjoint equation is a widely concerned issue in the field of optimal control (e.g., \cite{[P+90],[WU+17]}).

\appendix
\section*{Appendix}
%\section{Pairwise comparison and trajectories of four systems}
\begin{lemma}\label{ergodic dynamical system}
\textbf{(\cite{[D+15]})}
The Wiener shift $\theta_t:\Omega\rightarrow\Omega,$ $\theta_t W_s\triangleq W_{t+s}-W_t,~\forall~t,s\in[0,T]$ is a measurable, measure-preserving and ergodic dynamical system with invariant measure $\mathbb{P}_W$.
\end{lemma}
\begin{proof} For all $s,t,r\in[0,T]$, $\omega\in\Omega$, $A\in\mathcal{F}$, we need to verify that the following conditions are satisfied:

(i) Identity property: $\theta_0=Id$.\\
$\theta_0 W_s\triangleq W_{s}-W_0=W_{s}$, then $\theta_0=Id$.

(ii) Flow property: $\theta_{t+s}=\theta_t\circ\theta_s$.
$$\begin{aligned}
\theta_{t+s}W_r&=W_{t+s+r}-W_{t+s}\\
&=W_{t+s+r}-W_t+W_t-W_{t+s}\\
&=\theta_tW_{s+r}-\theta_tW_{s}\\
&=\theta_t(\theta_s W_r).
\end{aligned}$$
Thus, $\theta_{t+s}=\theta_t\circ\theta_s$.

(iii) Measurability property: The mapping $(\omega,t)\rightarrow\theta_t\omega$ is measurable.\\
In fact, $\theta_t$ is a homeomorphism for each $t$ and $(\omega,t)\rightarrow\theta_t\omega$ is continuous, hence measurable.

(iv) Measure-preserving property: $\mathbb{P}_W(\theta_tA)=\mathbb{P}_W(A)$.\\
For every fixed $t\in[0,T]$, a new probability measure $\theta_t\mathbb{P}_W$ on $\Omega$ is defined by $\theta_t\mathbb{P}_W(A)=\mathbb{P}_W(\theta_t^{-1}A)$. The mapping $\theta_t$ is measure-preserving if $\theta_t\mathbb{P}_W=\mathbb{P}_W$. Because $\theta_t^{-1}=\theta_{-t}$,
$$\mathbb{P}_W(\theta_tA)=\theta_t\mathbb{P}_W(\theta_tA)\triangleq\mathbb{P}_W
(\theta_t^{-1}\theta_tA)=\mathbb{P}_W(A).$$
The Wiener measure $\mathbb{P}_W$ is invariant and ergodic under $\theta_t$.
\end{proof}

\begin{lemma}\label{Prokhorov's theorem}
\textbf{(Prokhorov's theorem)}
Let $(\Omega,\mathcal{F},\mathbb{P})$ be a probability space, $\Omega$ be a Polish space, and $\mathcal{M}$ be a collection of measures defined on $\mathcal{F}$. Then $\mathcal{M}$ is tight if and only if the closure of $\mathcal{M}$  is sequentially compact in the space $(\Omega,\mathcal{F},\mathbb{P})$ equipped with the topology of weak convergence.
\end{lemma}

\begin{lemma}\label{Lyapunov function theorem}
\textbf{(\cite{[M+99]})}
Let $p,\lambda,c_1,c_2$ be positive numbers. Assume that there exists a function $V(x,t)\in C^{2,1}(\R_0^n \times\R^+;\R^+)$ such that $$c_1|x|^p\leq V(x,t)\leq c_2|x|^p$$ and
$$LV(x,t)\leq-\lambda|x|^p$$
for all $(x,t)\in\R_0^n\times\R^+$. Then
$$\underset{t\rightarrow\infty}{\lim\sup}\frac{1}{t}\log(\mathbb{E}|x(t,x_0)|^p)\leq-\frac{\lambda}{c_2}$$
for all $x_0\in\R^n$.
\end{lemma}

\begin{lemma}\label{Ito formula}
\textbf{(It\^{o} formula)}
For all $0\leq t\leq T$, the following It\^{o} formula holds:
$$\begin{aligned}
||X(t)||^2_{\mathcal{X}}=||X(0)||^2_{\mathcal{X}}&+\int_0^t(2\langle f(s,X_s),X_s\rangle_{\mathcal{X}}+||\sigma(s,X_s)||^2_{\mathcal{L}^2})ds\\
&+2\int_0^t\langle X_s,\sigma(s,X_s)dW_s\rangle_{\mathcal{X}},\\
||Y(t)||^2_{\mathcal{X}}=||Y(0)||^2_{\mathcal{X}}&+\int_0^t(2\langle g(s,Y_s),Y_s\rangle_{\mathcal{X}}+||\varsigma(s,Y_s)||^2_{\mathcal{L}^2})ds\\
&+2\int_0^t\langle Y_s,\varsigma(s,Y_s)dW_s\rangle_{\mathcal{X}}.
\end{aligned}
$$
\end{lemma}
\begin{lemma}\label{Gronwall lemma}
\textbf{(Gronwall lemma)}
Let $\alpha(t),\beta(t),u(t)$ be real-valued functions defined on $[0,T]$. Assume that $\beta(t),u(t)$ are continuous and that the negative part of $\alpha(t)$ is integrable on every closed and bounded subinterval of $[0,T]$. \\

(i) If $\beta(t)$ is non-negative and if $u(t)$ satisfies the integral inequality
$$u(t)\leq\alpha(t)+\int_0^t\beta(s)u(s)ds,~\forall~t\in[0,T],$$
then
$$u(t)\leq\alpha(t)+\int_0^t\alpha(s)\beta(s)\exp\bigg(\int_s^t\beta(r)dr\bigg)ds,~\forall~t\in[0,T].$$

(ii) If, in addition, the function $\alpha(t)$ is non-decreasing, then
$$u(t)\leq\alpha(t)\exp\bigg(\int_0^t\beta(s)ds\bigg),~\forall~t\in[0,T].$$
\end{lemma}

\begin{lemma}\label{Young inequality}
\textbf{(Young inequality)}
Let $p>0,q>0$, and $\frac{1}{p}+\frac{1}{q}=1$. When $1<p<+\infty$,
$$ |ab|\leq\frac{|a|^p}{p}+\frac{|b|^q}{q}.$$
Particularly, if $p=q=2$, the above inequality is also known as Cauchy inequality.
Further, let $\epsilon>0$. Replace $a$ and $b$ with $\epsilon^{\frac{1}{p}}a$ and $\epsilon^{-\frac{1}{p}}b$,
$$ |ab|\leq\frac{\epsilon |a|^p}{p}+\frac{\epsilon^{-\frac{q}{p}}|b|^q}{q}.$$
Particularly, if $p=q=2$, the above inequality is changed into
$$|ab|\leq\frac{\epsilon}{2}|a|^2+\frac{1}{2\epsilon}|b|^2.$$

When $0<p<1$, the inequality sign $``\leq"$ is inversed into $``\geq"$, and only when $|b|=|a|^{p-1}$, the $``="$ holds.
\end{lemma}
%\begin{lemma}\label{Fubini Theorem}
%\textbf{(Fubini Theorem)}
%\end{lemma}

 \end{document}